\documentclass[a4paper,12pt]{article}
\usepackage[english]{babel}
\usepackage[utf8]{inputenc}
\usepackage{amsfonts,amssymb}
\usepackage{amsmath}
\usepackage{mathtools}
\usepackage{mathrsfs}
\usepackage{amsthm}
\usepackage{float}
\usepackage{cite}
\usepackage{csquotes}
\usepackage{physics}
\usepackage{subfig}
\usepackage{url}
\usepackage[final]{graphicx}
\usepackage[hidelinks]{hyperref}
\usepackage{titlesec}
\usepackage[titletoc,toc,title]{appendix}
\theoremstyle{plain}
\newtheorem{theorem}{Theorem}
\newtheorem{prop}{Proposition}[section]

\newtheorem{lemma}[prop]{Lemma}

\theoremstyle{definition}
\newtheorem{remark}[prop]{Remark}

\newtheorem{example}[prop]{Example}
\DeclareMathOperator{\sign}{sign}
\DeclareMathOperator{\win}{wind}
\DeclareMathOperator{\col}{color}

\newcommand{\D}{\mathcal{D}}
\newcommand{\tl}{\mathbf{t}}
\newcommand{\bl}{\mathbf{b}}
\newcommand{\R}{\mathcal{R}}
\newcommand{\pl}{\mathbf{p}}
\newcommand{\prob}{\mathbb{P}}
\newcommand{\overbar}[1]{\mkern 1.5mu\overline{\mkern-1.5mu#1\mkern-1.5mu}\mkern 1.5mu}

\begin{document}
\title{3D domino tilings: irregular disks \\
and connected components under flips}
\author{Raphael de Marreiros}
\maketitle
\begin{abstract}

We consider three-dimensional domino tilings of cylinders $\R_N = \D \times [0,N]$ where $\D \subset \mathbb{R}^2$ is a fixed quadriculated disk and $N \in \mathbb{N}$.
A domino is a $2 \times 1 \times 1$ brick.
A flip is a local move in the space of tilings $\mathcal{T}(\R_N)$: remove two adjacent dominoes and place them back after a rotation.
The twist is a flip invariant which associates an integer number to each tiling.
For some disks $\D$, called \emph{regular}, two tilings of $\R_N$ with the same twist
can be joined by a sequence of flips once we add vertical space to the cylinder.
We have that if $\D$ is regular then the size of the largest connected component under flips of $\mathcal{T}(\R_N)$ is $\Theta(N^{-\frac{1}{2}}|\mathcal{T}(\R_N)|)$.
The domino group $G_\D$ captures information of the space of tilings.
A disk $\D$ is regular if and only if $G_{\D}$ is isomorphic to $\mathbb{Z} \oplus \mathbb{Z}/(2)$; sufficiently large rectangles are regular.

We prove that certain families of disks are irregular.
We show that the existence of a bottleneck in a disk $\D$ often implies irregularity.
In many, but not all, of these cases, we also prove that $\D$ is \emph{strongly irregular}, i.e., that there exists a surjective homomorphism from $G_{\D}^+$ (a subgroup of index two of $G_{\D}$) to the free group of rank two.
Moreover, we show that if $\D$ is strongly irregular then the cardinality of the largest connected component under flips of $\mathcal{T}(\R_N)$ is $O(c^N |\mathcal{T}(\R_N)|)$ for some $c \in (0,1)$.
\end{abstract}
\let\thefootnote\relax
\section{Introduction}

Domino tilings have been the subject of many studies, specially due to their connections with various topics such as dimer models.
In particular, much is known regarding two-dimensional domino tilings.
In this context, we consider connected regions in $\mathbb{R}^2$ formed by a finite union of closed unit squares with vertices in $\mathbb{Z}^2$.
A~\emph{domino} is then a union of two unit squares sharing an edge (i.e., a rectangle with sides of length two and one) and a \emph{domino tiling} of a given region is a covering composed of dominoes with disjoint interiors.
Numerous results have been established in two dimensions, including methods to decide whether a region admits a tiling~\cite{CL90,THR90} and to count the number of domino tilings~\cite{Kas61}.

Particularly relevant to this paper is the problem of connectivity by flips.
A~\emph{flip} is a local move that consists of a $90^{\circ}$ rotation of two adjacent parallel dominoes.
Thurston~\cite{THR90} proved that any two tilings of a \emph{quadriculated disk} (i.e., a planar~region homeomorphic to a closed disk) can be joined by a sequence of flips.
For non simply connected regions, the idea of flux of a tiling is presented in~\cite{STRD95}.
The flux is a flip invariant such that two tilings of a planar region can be joined by a sequence of flips if and only if they have the same flux.

The notion of domino tilings and the associated questions can be easily extended to higher dimensions.
However, many of the methods that work in two dimensions do not apply in higher dimensions, and much less is known in these cases.
We focus on dimension three, so that a region is as a connected union of closed unit cubes and a \emph{(3D) domino} is a parallelepiped formed by two unit cubes sharing a face.
In contrast to the two-dimensional case, the space of tilings of many simply connected regions is no longer connected by flips; see Figure~\ref{fig:isolatedtilings} for examples of tilings of $[0,4]^2 \times [0,2]$ which admit no flip.
The counting tilings problem becomes computationally more complex in dimension three~\cite{PY13}.

We investigate domino tilings of cylinders, a three-dimensional region obtained from a quadriculated disk.
Specifically, a \emph{cylinder} $\R_N \subset \mathbb{R}^3$ is a cubiculated region formed by the cartesian product of a quadriculated disk $\D \subset \mathbb{R}^2$ and an interval $[0,N]$ for some $N \in \mathbb{N}$.
We adopt the approach used in~\cite{MS15} and draw a tiling of $\R_N$ by describing its behavior at each floor $\D \times [K-1,K]$; for instance, see Figure~\ref{fig:comodesenhartiling}.
The set of tilings of $\R_N$ is denoted by $\mathcal{T}(\R_N)$.
In general, we consider cylinders where the underlying disk is balanced and nontrivial.
A disk $\D$ is \emph{balanced} if it contains an equal number of black and white unit squares; a unit square $[a,a+1]\times[b,b+1] \subset \D$ with $(a,b) \in \mathbb{Z}^2$ is \emph{white} if $a+b$ is even and \emph{black} if $a+b$ is odd.
Additionally, $\D$ is \emph{trivial} if its unit squares are adjacent (i.e., share an edge) to at most other two unit squares.

\begin{figure}[H] 
\centerline{
\includegraphics[width=0.98\textwidth]{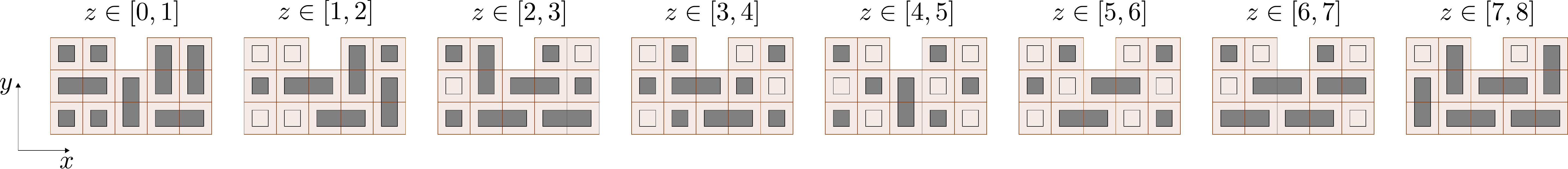}}
\caption{A tiling of a cylinder $\D \times [0,8]$.
The $x$-axis and the $y$-axis are fixed and floors are exhibited in increasing order from the left to the right.
Dominoes parallel to either the $x$-axis or the $y$-axis are represented as planar dominoes.
Vertical dominoes, i.e., dominoes parallel to $z$-axis, are represented by two unit squares contained in adjacent floors; to avoid confusion the unit square contained in the highest floor, which appears at the right hand side, is left unfilled.
}
\label{fig:comodesenhartiling}
\end{figure}

Milet and Saldanha \cite{MS18} introduced the \emph{twist} of a tiling for several contractible cubiculated regions contained in $\mathbb{R}^3$.
In this case, the twist is a flip invariant assuming values in $\mathbb{Z}$: given a tiling $\tl$ of a suitable region, we have an integer $\textsc{Tw}(\tl)$.
The definition of the twist involves a relatively complicated combinatorial formula.
Intuitively, the twist measures how twisted a tiling is by counting certain pairs of dominoes in different directions.

For a fixed nontrivial balanced quadriculated disk $\D \subset \mathbb{R}^2$ the distribution of the twist (according to the number of tilings of $\D \times [0,N]$) tends to a Gaussian as $N$ goes to infinity (see \cite{Sal21}).
In the same scenario, almost always any two tilings with the same twist can be joined by a sequence of flips.
Experimental evidence suggests that similar results hold for large cubical boxes.

In order to study the problem of connectivity by flips of the space of tilings of cylinders, we consider two distinct, but related, equivalence relations.
Consider a disk $\D$ and two tilings $\tl_1 \in \mathcal{T}(\R_{N_1})$ and $\tl_2 \in \mathcal{T}(\R_{N_2})$.
We write $\tl_1 \approx \tl_2$ if $N_1= N_2$ and there exists a sequence of flips joining $\tl_1$ and $\tl_2$.
We need two concepts to define the second equivalence relation.
First, let the concatenation $\tl_1 * \tl_2$ be the tiling of $\R_{N_1+N_2}$ formed by the union of $\tl_1$ and the translation of $\tl_2$ by $(0,0,N_1)$.
Second, for $N \in 2\mathbb{N}$ let the \emph{vertical tiling} $\tl_{\text{vert},N} \in \mathcal{T}(\R_N)$ be the tiling consisting solely of vertical dominoes.
We say that $\tl_1 \sim \tl_2$ if $N_1 \equiv N_2 \pmod 2$ and there exist $M_1,M_2 \in 2\mathbb{N}$ such that $N_1+M_1=N_2+M_2$ and $\tl_1*\tl_{\text{vert},M_1} \approx \tl_2*\tl_{\text{vert},M_2}$.
Figure~\ref{fig:isolatedtilings} illustrates two tilings $\tl_1$ and $\tl_2$ of $[0,4]^2 \times [0,2]$ which admit no flip, implying that $\tl_1 \not\approx \tl_2$.
However, as can be verified through a construction, $\tl_1*\tl_{\text{vert},2} \approx \tl_2*\tl_{\text{vert},2}$, so that $\tl_1 \sim \tl_2$.
\begin{figure}[H] 
\centerline{
\includegraphics[width=0.6\textwidth]{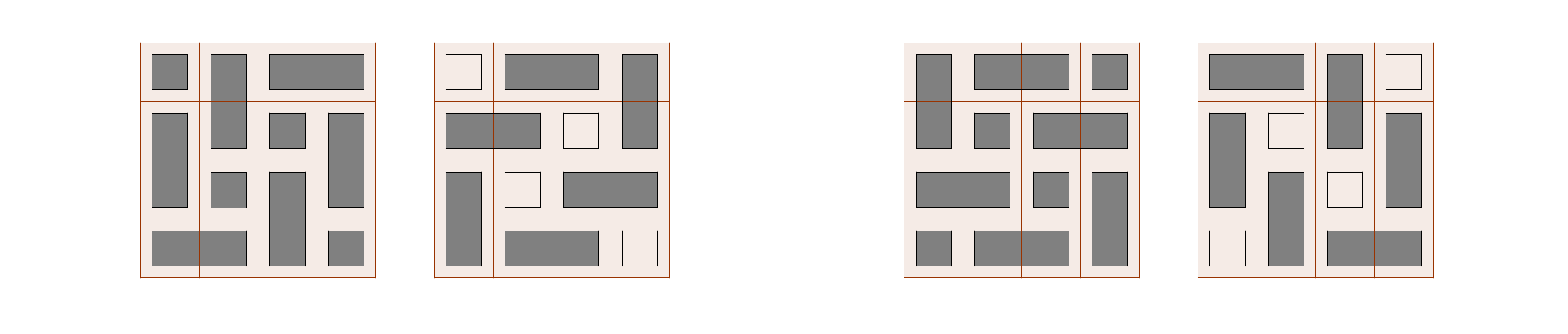}}
\caption{Two tilings $\tl_1$ and $\tl_2$ of $[0,4]^2 \times [0,2]$.}
\label{fig:isolatedtilings}
\end{figure}

The concept of regular disk is presented in~\cite{Sal22}.
A nontrivial balanced quadriculated disk $\D$ is \emph{regular} if whenever two tilings $\tl_1$ and $\tl_2$ of $\D \times [0,N]$ have the same twist then $\tl_1$ and $\tl_2$ can be connected by a sequence of flips provided that some vertical space is allowed; more precisely, if $\textsc{Tw}(\tl_1)=\textsc{Tw}(\tl_2)$ then $\tl_1 \sim \tl_2$.
We say that a disk is \emph{irregular} if it is not regular.
Saldanha \cite{Sal22} proved that the rectangle $\D = [0,L] \times [0,M]$ with $LM$ even is regular if and only if $\min \{ L,M \} \geq 3$; and it was conjectured that ``plump'' disks are regular.
For instance, Figure~\ref{fig:regcounterexample} below exhibits examples of regular disks.

\begin{figure}[H]
\centerline{
\includegraphics[width=0.66\textwidth]{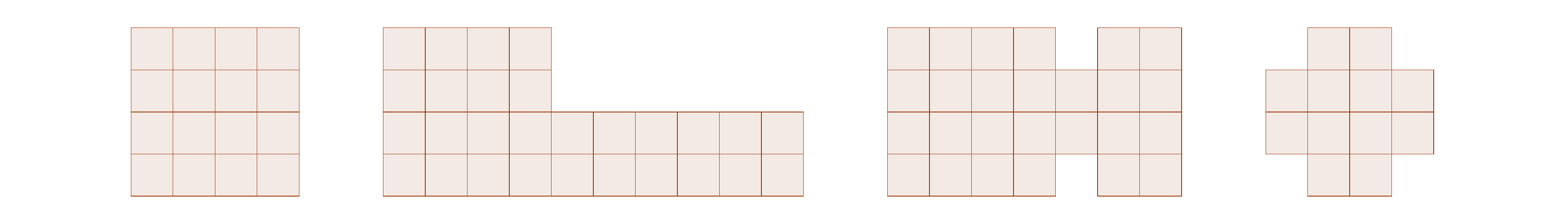}}
\caption{Examples of disks whose regularity follows from \cite{Sal22}, \cite{Mar21} and brute force.}
\label{fig:regcounterexample}
\end{figure}

In this paper, we prove that other families of disks are irregular.
We show that, for some disks $\D$, the existence of either a unit square or a domino that disconnects $\D$ implies that $\D$ is irregular.
As a consequence, the disks in Figure~\ref{fig:nonreg} are irregular.

\begin{figure}[H] 
\centerline{
\includegraphics[width=0.65\textwidth]{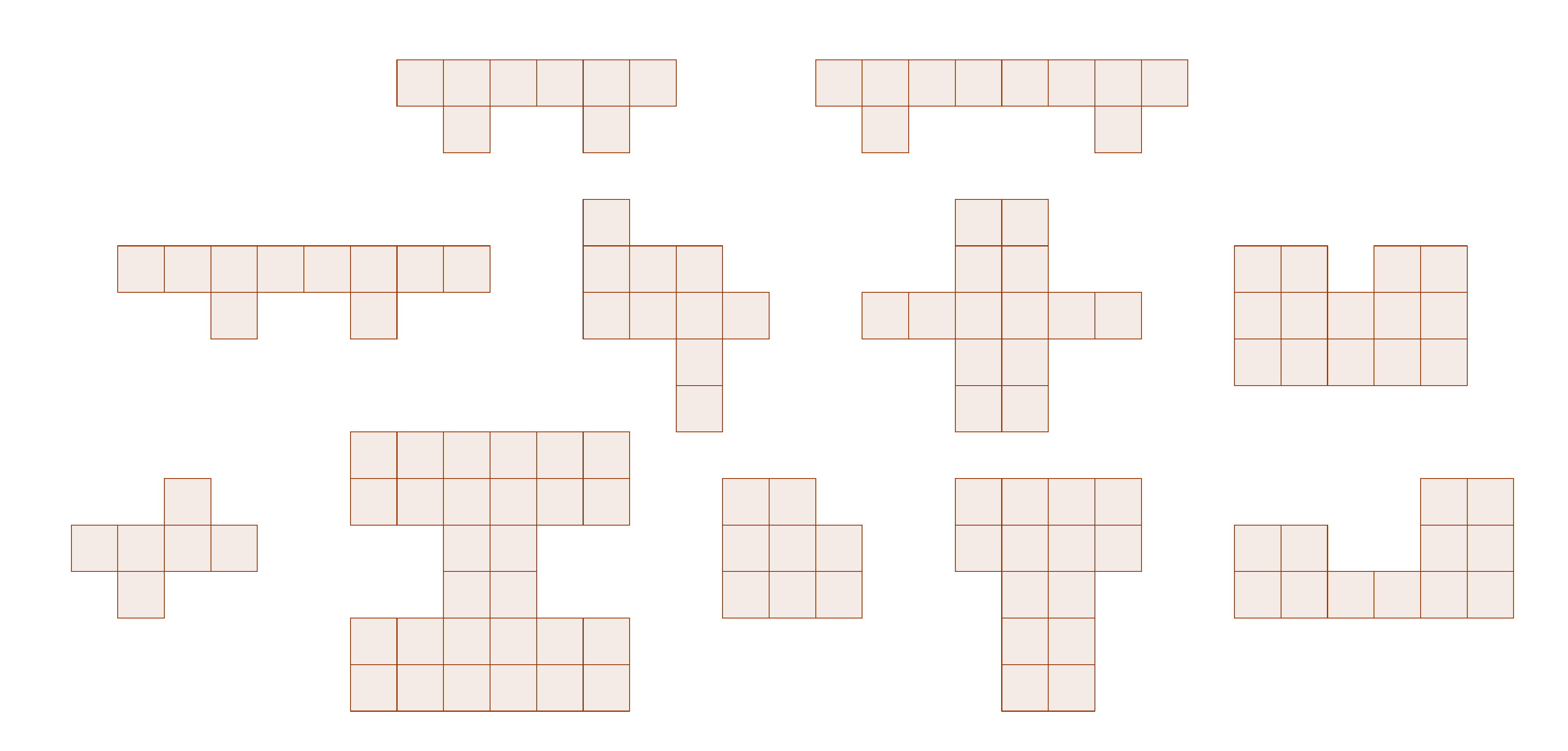}}
\caption{Examples of irregular disks.}
\label{fig:nonreg}
\end{figure}

We study the irregularity of a disk $\D$ by considering its domino group $G_\D$.
The domino group is the fundamental group of a finite 2-dimensional CW-complex $\mathcal{C}_\D$, whose construction is described in \cite{Sal22} and in Section~\ref{sec:definitions}.
For the moment, we briefly overview its construction.
The key idea is to associate each tiling of $\R_N$ with an oriented closed path of length $N$ in the 1-skeleton of $\mathcal{C}_\D$.
A tiling of $\R_N$ is drawn as a sequence of $N$ floor configurations.
We convert this sequence into an oriented path in $\mathcal{C}_\D$ by assigning an oriented edge to each floor configuration, with the initial (resp.\ final) vertex of an edge corresponding to the unit squares occupied by vertical dominoes that intersect the preceding (resp.\ succeeding) floor.
The 2-cells are then attached to attend flips.
Thus, two tilings are equivalent under $\sim$ if and only if their associated closed paths are homotopic.

The even domino group $G_{\D}^+$ is the normal subgroup of index two of $G_\D$ whose elements correspond to tilings of cylinders of even height.
In general, $G_\D^+$ carries most of the relevant information about the domino group.
If there exists a domino tiling of $\D \times [0,1]$ then $G_{\D}$ is isomorphic to a semidirect product of $G_{\D}^+$ and $\mathbb{Z}/(2)$.  
The twist defines a homomorphism from $G_{\D}^+$ to the integers $\mathbb{Z}$; it turns out that $\D$ is regular if and only if this homomorphism is an isomorphism.
In this case, $G_{\D} =\mathbb{Z} \oplus \mathbb{Z}/(2)$.

We distinguish irregular disks by the behavior of their even domino groups.
A disk $\D$ is called \emph{strongly irregular} if there exists a surjective homomorphism from the even domino group $G_{\D}^+$ to $F_2 =\langle a,b\rangle$, the free group generated by $a$ and $b$ with $e$ as the identity.
The structure of $G_\D^+$ provides information about the space of tilings $\mathcal{T}(\R_N)$ for large values of $N$.
The $\approx$-equivalence classes in $\mathcal{T}(\R_N)$ are called connected components under flips.
In the regular case, i.e. $G_{\D}^+ = \mathbb{Z}$, it follows from \cite{Sal21} that the size of the largest component is $\Theta(N^{-\frac{1}{2}}|\mathcal{T}(\R_N)|)$.
We prove that if $\D$ is strongly irregular then, for large values of $N$, the connected components under flips consist of exponentially small fractions of $\mathcal{T}(\R_N)$.

\begin{theorem}\label{thm:sizecomp}
Consider a nontrivial balanced quadriculated disk $\D$.
Let $\mathbf{T}_1$ and $\mathbf{T}_2$ be two independent random tilings of $\D \times [0,N]$.
If $\D$ is strongly irregular then there exists $c \in (0,1)$ such that $\mathbb{P} (\mathbf{T}_1 \approx \mathbf{T}_2) = o(c^N)$.
\end{theorem}

\begin{remark}
It follows from~\cite{Sal21} that the probability $\mathbb{P}(\textsc{Tw}(\mathbf{T}_1) = \textsc{Tw}(\mathbf{T}_2))$ is asymptotically bounded below by $N^{-1}$ for any nontrivial balanced disk $\D$.
Thus, Theorem~\ref{thm:sizecomp} implies that $\mathbb{P}(\mathbf{T}_1 \approx \mathbf{T}_2 \mid \textsc{Tw}(\mathbf{T}_1)=\textsc{Tw}(\mathbf{T}_2)) = o(c^N)$ if $\D$ is a strongly irregular disk. 
\end{remark}

Normally, computing the even domino group of a disk is harder than proving its strongly irregularity.
In~\cite{Sal22}, it is proved that thin rectangles $\D_L= [0,L] \times [0,2]$ with $L \geq 3$ are strongly irregular.
Inspired by the ideas and results presented in~\cite{MS15}, we compute the even domino group $G_{\D_L}^+$.
Let $S_L= \{ a_i \colon i \in \mathbb{Z}_{\neq 0}\text{  and } |i| \leq \genfrac{\lfloor}{\rfloor}{0.5pt}{1}{L-1}{2} \}$ be a set of symbols and consider $R_L = \{ (m,n) \in \mathbb{Z}^2 \colon \max\{|m|,|n|,|m-n|\} < \genfrac{\lfloor}{\rfloor}{0.5pt}{1}{L}{2} \}$.
Now, consider the group described in terms of generators and relations
\begin{align}\tag{1.1}
G_L^+ = \langle S_L \mid [a_m, a_n]=1 \text{ for } (m,n) \in R_L \rangle.
\label{eq:1.1}
\end{align}
In other words, $G_L^+$ is the quotient of the free group on $S_L$ by the normal subgroup generated by $[a_m, a_n]$ with $(m,n) \in R_L$.

\begin{theorem}\label{thm:thinrectangles}
Let $L\geq 3$ and consider the disk $\D_L = [0,L] \times [0,2]$.
Then, the even domino group $G_{\D_L}^+$ is isomorphic to $G_L^+$.
\end{theorem}

In most cases, we demonstrate the strong irregularity of a disk $\D$ by explicitly constructing a surjective homomorphism from $G_{\D}^+$ to $F_2$.
This construction is based on the existence of floor configurations composed of dominoes in a staggered configuration, as the fifth floor of the tiling exhibited in Figure~\ref{fig:comodesenhartiling}.
Although obtaining such disks is not too hard, their description can be somewhat intricate.
For instance, the disks in Figure~\ref{fig:irregdisk-squrdisc} are strongly irregular by Theorem~\ref{thm:irregdisk-squrdisc}.
Additional constructions of strongly irregular disks are discussed in Section~\ref{sec:proofthms123}.

\begin{theorem}\label{thm:irregdisk-squrdisc}
Consider a balanced quadriculated disk $\D$.
Suppose that $\D$ contains a unit square $s$ such that $\D \smallsetminus s$ has at least three connected components.
If the largest component of $\D \smallsetminus s$ has size at most $|\D|-4$ then $\D$ is strongly irregular.
\end{theorem}
\begin{figure}[H]
\centerline{
\includegraphics[width=0.61\textwidth]{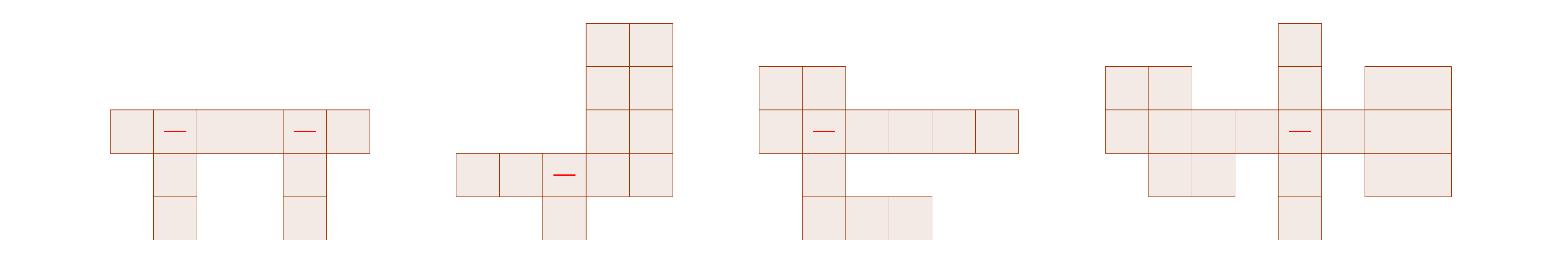}}
\caption{Examples of strongly irregular disks; unit squares $s$ as in Theorem~\ref{thm:irregdisk-squrdisc} are marked by a red line segment.}
\label{fig:irregdisk-squrdisc}
\end{figure}

Disks that are irregular but not strongly irregular seem to be relatively rare.
Indeed, the only family of disks for which we demonstrate this phenomenon is illustrated in the first row of Figure~\ref{fig:nonreg}.
Notice that these disks can be disconnected into three connected components by removing a unit square, with the largest component having size $|\D|-3$.

The text is divided as follows.
In Section~\ref{sec:definitions} we review the construction of the domino group (see \cite{Sal22}) and show in Example~\ref{exm:diskz2} the existence of disks which are irregular but not strongly irregular.
The proof of Theorem~\ref{thm:sizecomp} is presented in Section~\ref{sec:proofthmcomp}.
We compute the domino group of thin rectangles $[0,L] \times [0,2]$ with $L \geq 3$ in Section~\ref{sec:thinrec}.
Section~\ref{sec:proofthms123} is dedicated to the study of strongly irregular disks.

The author thanks Nicolau Saldanha for providing insightful ideas, comments and suggestions.
Acknowledgments are also given to Caroline Klivans, Robert Morris, Simon Griffiths and Yoshiharu Kohayakawa who carefully read the author's Master dissertation~\cite{Mar21}, which corresponds to part of this paper.
The author thanks the referees for the thorough reading and helpful remarks that greatly improved this work.
The support of CNPq, CAPES, FAPERJ and Projeto Arquimedes (PUC-Rio) are appreciated.
\section{Definitions}\label{sec:definitions}

Henceforth, unless otherwise stated, we always assume that quadriculated disks are nontrivial and balanced.
Consider a quadriculated disk $\D$.
A \emph{plug} $p$ is a union of an equal number of white and black unit squares contained in $\D$.
In particular, we have the empty plug $\pl_{\circ}=\emptyset$ and the full plug $\pl_{\bullet}=\D$.
The complement $\D \smallsetminus \text{int}(p)$ of a plug $p$ is also a plug and is denoted by $p^{-1}$; notice that $\pl_{\circ}^{-1} = \pl_{\bullet}$. 
The number of unit squares in $p$ is denoted by $|p|$.
Let $\mathcal{P}$ be the set of plugs in $\D$.
We say that two plugs $p_1,p_2 \in \mathcal{P}$ are \emph{disjoint} if their interiors are disjoint.

Sometimes, it is useful to consider a region more general than cylinders.
Let $p_1, p_2 \in \mathcal{P}$ be two plugs and consider two nonnegative integers $N_1$ and $N_2$ such that $N_2 > N_1 +2$.
The \emph{cork} $\R_{N_1,N_2;p_1,p_2}$ is defined as:
$$\R_{N_1,N_2;p_1,p_2} = (\D \times [N_1+1,N_2-1]) \cup (p_1^{-1} \times [N_1, N_1 + 1]) \cup (p_2^{-1} \times [N_2-1,N_2]).$$
In other words, the cork $\R_{N_1,N_2;p_1,p_2}$ is obtained from $\D \times[N_1,N_2]$ by removing the plug $p_1$ from the $(N_1+1)$-th floor and the plug $p_2$ from the $N_2$-th floor.
For instance, notice that $\R_{0,N;\pl_{\circ},\pl_{\circ}} = \R_N$.
The inverse of a tiling $\tl$ of $\R_{N_1,N_2; p_1, p_2}$ is defined as the tiling $\tl^{-1}$ of $ \R_{N_1,N_2; p_2,p_1}$ obtained by reflecting $\tl$ on the $xy$ plane.
If $\tl \in \mathcal{T}(\R_N)$ then $\tl * \tl^{-1} \approx \tl_{\text{vert}, 2N}$ (see Lemma 4.1 of~\cite{Sal22}).

Given a quadriculated disk $\D$, we construct a 2-complex $\mathcal{C}_\D$.
The domino group $G_\D$ will be the fundamental group of $\mathcal{C}_\D$.
The 0-skeleton, the set of vertices of $\mathcal{C}_\D$, is the set of plugs $\mathcal{P}$.
The edges of $\mathcal{C}_\D$ represent two-dimensional domino tilings of subregions of $\D$.
More precisely, attach an edge between two disjoint plugs $p_1$ and $p_2$ for each tiling of $\D \smallsetminus \text{int}(p_1 \cup p_2)$.
In particular, plugs that are not disjoint do not share any edge.
For the special case $p_1=p_2 = \pl_{\circ}$ (the empty plug), there is a loop based at $\pl_{\circ}$ for each tiling of $\D$; there are no other loops in $\mathcal{C}_\D$.
This construction defines the 1-skeleton of $\mathcal{C}_\D$.

We now attach the 2-cells.
First, attach a disk to each loop by wrapping its boundary twice around the loop, so that the result is a projective plane.
In other words, each loop is homotopically equivalent to a circle $f$, and we attach a disk via the map $f*f$.
For instance, see Figure~\ref{fig:loops}. 

\begin{figure}[H] 
\centerline{
\includegraphics[width=0.58\textwidth]{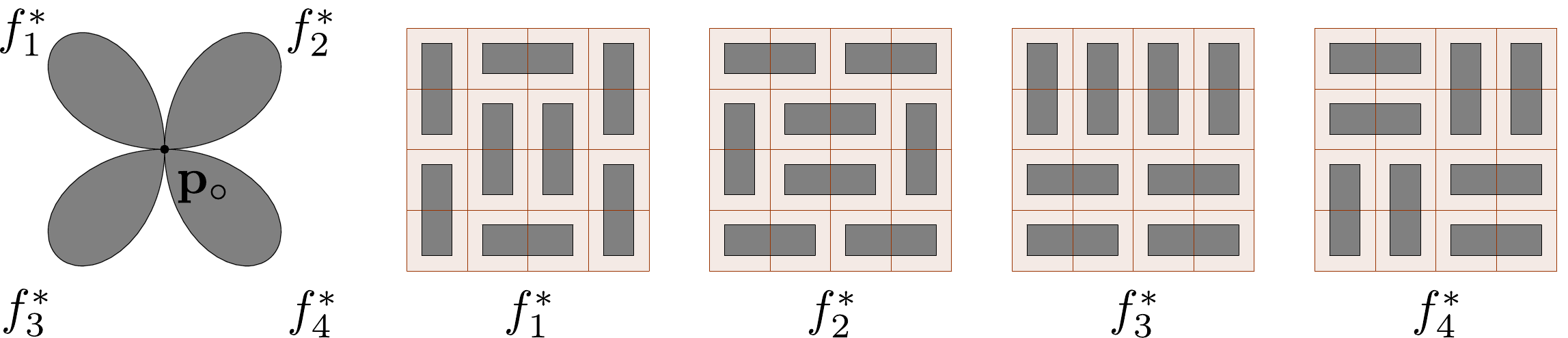}}
\caption{Four loops in $\mathcal{C}_\D$ for $\D = [0,4]^2$, the four petals are four projective planes.}
\label{fig:loops}
\end{figure}

The other 2-cells are attached injectively to certain bigons and quadrilaterals.
Notice that a bigon consists of two distinct disjoint plugs $p_1$ and $p_2$ connected by two edges representing distinct tilings of $\D \smallsetminus \text{int}(p_1 \cup p_2)$.
We attach a disk to each bigon whose edges correspond to tilings which differ by a single flip, as in Figure~\ref{fig:horflip}.

\begin{figure}[H] 
\centerline{
\includegraphics[width=0.57\textwidth]{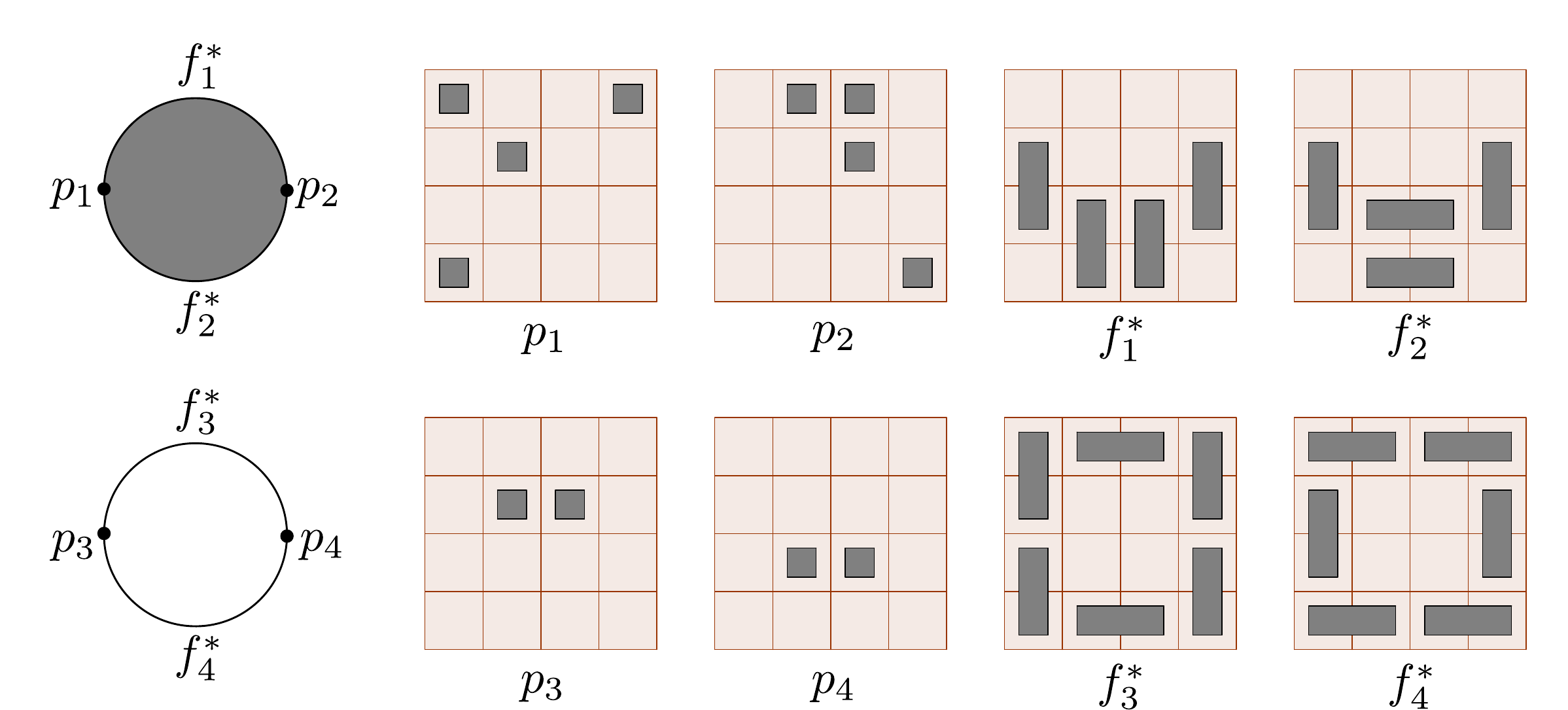}}
\caption{Two bigons in the complex of $\D=[0,4]^2$.
In the first row we attach a disk to the bigon.
The two tilings in the second row do not differ by a single flip.
Therefore, we do not attach a disk to the bigon in the second row.}
\label{fig:horflip}
\end{figure}

We attach a 2-cell to each quadrilateral constructed in the following manner. 
Let $p_1,p_2,p_3,p_4$ be four plugs such that $p_2$ is disjoint from $p_1$ and $p_4$.
Suppose that $p_2$ equals a union of $p_3$ and two adjacent unit squares, i.e. a domino $d \subset \D$.
Let $f_1^*$ be a tiling of $\D \smallsetminus \text{int}(p_1\cup p_2)$ and $f_2^*$ be a tiling of $\D \smallsetminus \text{int}(p_2 \cup p_4)$.
Therefore, $f_3^* = f_1^* \cup d$ and $f_4^* = f_2^* \cup d$ are tilings of $\D \smallsetminus \text{int}(p_1 \cup p_3)$ and $\D \smallsetminus\text{int}(p_3 \cup p_4)$, respectively.
Notice that $p_1,p_2,p_3,p_4$ and $f_1^*,f_2^*,f_3^*,f_4^*$ thus define a quadrilateral in the 1-skeleton of $\mathcal{C}_\D$.
Then, attach a 2-cell to this quadrilateral.
For instance, see Figure~\ref{fig:vertflip}.

\begin{figure}[H] 
\centerline{
\includegraphics[width=0.9\textwidth]{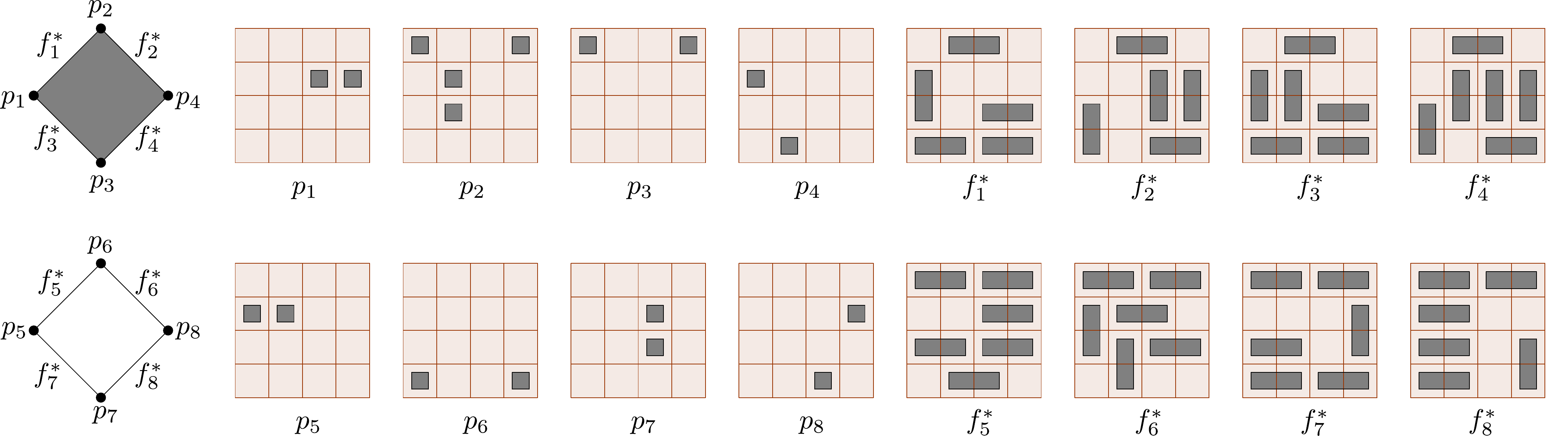}}
\caption{Two quadrilaterals in the complex of $\D=[0,4]^2$.
In the first row we attach a 2-cell to the quadrilateral.
The second row shows an example of a quadrilateral where we do not attach a 2-cell.}
\label{fig:vertflip}
\end{figure}

This finishes the construction of the complex $\mathcal{C}_\D$.
By definition, $G_\D = \pi_1(\mathcal{C}_\D, \pl_{\circ})$.
In most cases, it is impractical to draw the complex $\mathcal{C}_\D$.
For instance, if $\D = [0,4]^2$ then $\mathcal{C}_\D$ has 12870 vertices and 36 loops.
The calculation of the exact number of 1-cells and 2-cells requires a long computation.

The complex $\mathcal{C}_\D$ is related to tilings of cylinders of base $\D$.
Indeed, a tiling $\tl$ of $\D \times [0,N]$ corresponds to an oriented closed path of length $N$ based at $\pl_{\circ}$.
In order to construct this correspondence we proceed as follows.
We describe the behavior of $\tl$ at each floor $\D \times [K-1, K]$ by a triple $f_K= (p_{K-1},f_K^*,p_K)$.
The plug $p_K$ is the union of the unit squares $[a,a+1] \times [b,b+1]$ such that $[a,a+1]\times[b,b+1] \times [K-1,K+1]$ is contained in $\tl$.
The \emph{reduced $K$-th floor} $f_K^*$ is formed by the dominoes $d \subset \D$ such that $d \times [K-1,K]$ is contained in $\tl$.
Notice that the interiors of $p_{K-1}$ and $p_K$ are disjoint, and $f_K^*$ is a tiling of $\D \smallsetminus \text{int}(p_{K-1} \cup p_K)$.
Therefore, each triple $f_K$ corresponds to an oriented edge in $\mathcal{C}_\D$.
Thus, by identifying each floor with a triple, the tiling $\tl$ is described by a closed oriented path in $\mathcal{C}_\D$:
$$\tl = (\pl_{\circ},f_{1}^*,p_{1}) * (p_1,f_{2}^*,p_{2}) * \ldots * (p_{N-1},f_{N}^*,\pl_{\circ}).$$

In the complex $\mathcal{C}_\D$, we specify an orientation each time we move along an edge.
Consider two distinct disjoint plugs $p_1$ and $p_2$ and let $f_1^*$ be a tiling of $\D \smallsetminus \text{int}(p_1 \cup p_2)$.
Thus, $f_1^*$ is an edge of $\mathcal{C}_\D$.
We then denote the two possible orientations of $f_1^*$ by $f=(p_1, f_1^*,p_2)$ and $f^{-1}=(p_2, f_1^*, p_1)$.
Due to the connection between tilings and paths, oriented edges in $\mathcal{C}_\D$ are also called \emph{floors}.
Notice that, since projective planes are attached to loops, the two possible orientations of a loop $f=(\pl_{\circ}, f^*, \pl_{\circ})$ are homotopic.

Under this identification of tilings and paths, concatenation of paths in $\mathcal{C}_{\D}$ corresponds to concatenation of tilings.
In that sense, flips correspond to homotopies between paths.
Then, two tilings are equivalent under $\sim$ if and only if their corresponding paths in $\mathcal{C}_{\D}$ are homotopic (see Lemma 5.4 of~\cite{Sal22}).

The even domino group $G_\D^+$ is a subgroup of $G_\D$.
More precisely, $G_\D^+$ is the kernel of the homomorphism $\psi \colon G_\D \to \mathbb{Z}/(2)$ which takes a closed path of length $N$ to $N \mod 2$.
In other words, $G_\D^+$ consists of closed paths of even length.
Notice that $G_\D^+$ is a normal subgroup of index two of $G_\D$.

The even domino group is the fundamental group of a double cover $\mathcal{C}_{\D}^+$ of $\mathcal{C}_{\D}$, i.e., $\pi_1(\mathcal{C}_{\D}^+)=G_{\D}^+$.
The set of vertices of $\mathcal{C}_{\D}^+$ is the set $\mathcal{P} \times \mathbb{Z} \slash (2)$, which indicates the plug and the parity of its position.
Moreover, if $p_1$ and $p_2$ are two disjoint plugs then each tiling $f_1^*$ of $\D \smallsetminus (p_1 \cup p_2)$ corresponds to two edges in $\mathcal{C}_\D^+$.
More specifically, for each $i \in \mathbb{Z}/(2)$, there is an edge $f_{1,i}^*$ between $(p_1,i+1)$ and $(p_2,i)$.
Therefore, $\mathbf{f}_i=((p_1,i+1), f_{1,i}^*, (p_2,i))$ and  $\mathbf{f}_i^{-1}=((p_2,i), f_{1,i}^*,(p_1, i+1))$ define two orientations of $f_{1,i}^*$.
We prefer to describe the orientation of an edge in $\mathcal{C}_\D^+$ by a pair formed by an oriented edge in $\mathcal{C}_\D$ and an element $i \in\mathbb{Z}/(2)$.
The oriented edge of $\mathcal{C}_\D$ indicates the initial and the final vertex, the element of $\mathbb{Z}/(2)$ indicates the parity of the final vertex.
For instance, an oriented edge $\mathbf{f}_i=((p_1,i+1), f_{1,i}^*, (p_2,i))$ in $\mathcal{C}_\D^+$ is described by the pair $(f,i)$ where $f=(p_1,f_1^*,p_2)$ is an oriented edge in $\mathcal{C}_\D$.
Therefore, oriented edges in $\mathcal{C}_\D^+$ are also called \emph{floors with parity}.
Notice that if $\mathbf{f}_i=(f,i)$ then $\mathbf{f}_i^{-1}=(f^{-1},i+1)$.

We now briefly recall the definition of the twist of a tiling (see \cite{FKMS22, MS18, Sal22}).
The twist defines a homomorphism from the even domino group $G_{\D}^+$ to the integers $\mathbb{Z}$.
If $\D$ is nontrivial then this homomorphism is surjective (see Lemma 6.2 of~\cite{Sal22}).
In particular, the domino group of a nontrivial disk is infinite.

First, assume that the unit cubes in $\mathbb{R}^3$ are colored alternately in black and white.
For an arbitrary domino $d$, let $v(d) \in \{ \pm e_1, \pm e_2, \pm e_3 \} \subset \mathbb{R}^3$ be the unit vector from the center of the white cube to the center of the black cube of $d$.
For $u \in \{ \pm e_1, \pm e_2\}$ let $\mathcal{S}^u(d)$ be the interior of the set $(\bigcup_{t \in [0, \infty)} d+tu) \smallsetminus d$.

Given a tiling $\tl$ of $\D \times [0,N]$ and two dominoes $d_1$ and $d_2$ of $\tl$ let 
$$\tau^u(d_1,d_2) = 
\begin{cases}
\frac{1}{4}\det(v(d_2),v(d_1),u), & \quad d_2 \cap \mathcal{S}^u(d_1) \neq \emptyset \\
0, & \quad \text{otherwise.} \\
\end{cases} 
$$
Thus, $\tau^u(d_1,d_2)=0$ unless $d_1$ is a vertical (resp.\ horizontal) domino and $d_2$ is a horizontal (resp.\ vertical) domino such that $d_1$ and $d_2$ are contained in $\D \times [N_0,N_0+2]$ for some $N_0$.

The \emph{twist} of $\tl$ is defined as the sum:
$$\textsc{Tw}(\tl) = \sum_{d_1,d_2 \in \tl} \tau^u(d_1,d_2).$$
The twist is always an integer number and does not depend on the choice of $u$.
Additionally, the twist is invariant under flips.

We end this section by computing the even domino group of an important family of disks.
The disks in this family are irregular but not strongly irregular.

\begin{example}\label{exm:diskz2}
Consider $L \geq 3$.
Let $\D_L$ be the disk formed by the union of the rectangle $R_L = [0, 2L] \times [0, 1]$ and the two unit squares $s_0 = [1,2] \times [-1,0]$ and $s_L = [2L-2,2L-1] \times [-1,0]$; the disk $\D_3$ is shown in the top left corner of Figure~\ref{fig:nonreg}.
We claim that the even domino group $G_{\D_L}^+$ is isomorphic to $\mathbb{Z} \oplus \mathbb{Z}$.

Let $d_0 = [1,2] \times [-1,1]$ and $d_L= [2L-2,2L-1] \times [-1,1]$ be the dominoes that contain $s_0$ and $s_L$, respectively.
We say that a three-dimensional domino is \emph{nontrivial} if its projection on the $xy$-plane is equal to either $d_0$ or $d_L$.

By removing the nontrivial dominoes and omitting the dominoes contained in $(s_0 \cup s_L) \times [0,2N]$, a tiling of $\D_L \times\ [0,2N]$ can be viewed as a tiling of a planar subregion of $[0,2L] \times [0,2N]$; as shown in Figure~\ref{fig:irregz2}.
This perspective allows us to apply results regarding the connectivity under flips of two-dimensional tilings.
The terms horizontal and vertical retain their original sense.

\begin{figure}[H]
\centerline{
\includegraphics[width=0.55\textwidth]{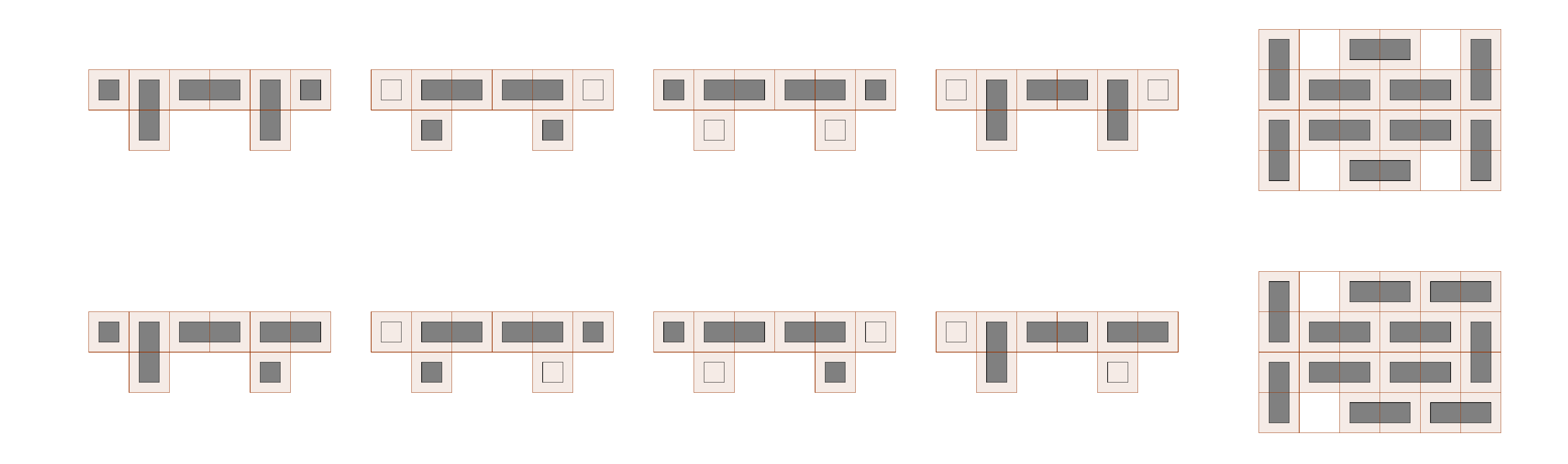}}
\caption{Two tilings of $\D_3 \times [0,4]$ and their depictions in a subregion of $[0,6]\times[0,4]$.
In this subregion, a domino parallel to the $x$-axis (resp.\ $y$-axis) corresponds to a three-dimensional horizontal (resp.\ vertical) domino.}
\label{fig:irregz2}
\end{figure}

We distinguish eight tilings $\tl_{L,1}, \tl_{L,2}, \ldots \tl_{L,8}$ of $\D_L \times [0,6]$.
For $L=3$, the two-dimensional perspective of these tilings are illustrated in Figure~\ref{fig:z2L3}.
For $L>3$, the tilings are derived from the case $L=3$ by translating the dominoes contained in $([4,6] \times [0,1] \times [0,6]) \cup ([4,5] \times [-1,0] \times [0,6])$ by $(2L-6,0,0)$, and by adding horizontal dominoes in $[4,2L-2] \times [0,1] \times [0,6]$. 
We first prove that these eight tilings generate the even domino group $G_{\D_L}^+$.
Then, we show that this family can be reduced to two tilings.

\begin{figure}[H]
\centerline{
\includegraphics[width=0.9\textwidth]{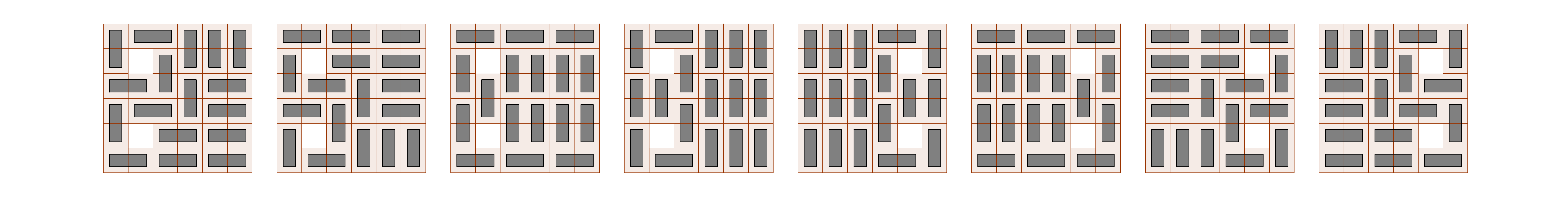}}
\caption{The 2D representation of eight tilings of $\D_3 \times [0,6]$.}
\label{fig:z2L3}
\end{figure}

Notice that the nontrivial dominoes in $\tl_{L,j}$ are separated by two floors.
This separation can be arbitrarily increased, as the relation $\sim$ allows the insertion of vertical floors between any two floors (see Lemma 5.3 of~\cite{Sal22}).
For any positive even $k$, let $\tl_{L,j,k}$ be the tiling obtained from $\tl_{L,j}$ by inserting $k-2$ vertical floors between the two nontrivial dominoes; in particular, $\tl_{L,j,2}$ is equal to $\tl_{L,j}$.
Thus, $\tl_{L, j} \sim \tl_{L,j,k}$ and the nontrivial dominoes in $\tl_{L,j,k}$ are separated by $k$ floors.

Consider an arbitrary fixed tiling $\tl$ of $\D_L \times [0,2N]$.
If $\tl$ does not contain nontrivial dominoes, then $\tl$ can be viewed as a planar tiling of $[0,2L] \times [0,2N]$.
Therefore, since any two tilings of a quadriculated disk can be joined by a sequence of flips, it follows that $\tl$ is equivalent to a vertical tiling.

Now, suppose $\tl$ contains nontrivial dominoes.
By following an approach similar to that in~\cite{Sal22} for obtaining generators of an even domino group, we add vertical floors, shift them downward, use flips to move nontrivial dominoes and decompose $\tl$ into a concatenation of tilings of even height, each containing exactly two nontrivial dominoes: $\tl \sim \tl_1*\tl_2*\ldots*\tl_{r}$.
Notice that the fact that $\tl_i$ has even height implies that the projections of its nontrivial dominoes coincide.

We proceed to show that each tiling $\tl_i$ is $\sim $-equivalent to $\tl_{L,j}$ for some $j$.
Without loss of generality, assume the nontrivial dominoes in $\tl_i$ are of the form $d_0 \times [K-1, K]$ and $d_0 \times [M-1, M]$, with $K < M$.
It suffices to consider the case $M-1-K \geq 2$.
Otherwise, the nontrivial dominoes are adjacent and therefore $\tl_i$ can be transformed into a tiling without nontrivial dominoes, which is equivalent to a vertical tiling.

Observe that, due to the shape of $\D_L$, the unit cubes $[0,1]^2 \times [K-1, K]$ and $[0,1]^2 \times [M-1, M]$ are covered by vertical dominoes $d_i$ and $\tilde{d_i}$, respectively.
There are four possible configurations, determined by whether the vertical dominoes covering these unit cubes also cover the adjacent unit cube in the floor above or below.
Choose $j$ such that the translation of $d_i$ and $\tilde{d_i}$ by $(0,0,2-K)$ is contained in $\tl_{L,j,M-1-K}$.

By adding vertical floors to the first and last floors, we can align the nontrivial dominoes in $\tl_i$ and $\tl_{L,j,M-1-K}$, so that both tilings now have the same height and contain the same nontrivial dominoes.
Consequently, they define two tilings of the same planar non simply connected region.
The choice of $j$ implies that these two tilings have the same flux (in the sense of~\cite{STRD95}).
Since and any two tilings with the same flux can be joined by a sequence of flips, it follows that $\tl_i \sim \tl_{L,j,M-1-K}$.

The previous paragraphs imply that the even domino group $G_{\D_L}^+$ is generated by the eight tilings $\tl_{L,j}$, we now reduce this family to two tilings.
Notice that the first two tilings in Figure~\ref{fig:z2L3} are inverses of each other, as are the last two tilings.
The remaining tilings are equivalent to the vertical tiling, since three flips transform them into a tiling that does not contain any nontrivial dominoes.
Thus, we conclude that $G_{\D_3}^+$ is generated by the first and last tilings, $\tl_3$ and $\tilde{\tl}_3$, respectively.
Similarly, for $L>3$, the tilings $\tl_L=\tl_{L,1}$ and $\tilde{\tl}_L=\tl_{L,8}$ generate $G_{\D_L}^+$.
A lengthy verification shows that $\tl_L$ and $\tilde{\tl}_L$ commute, i.e., $\tl_L*\tilde{\tl}_L \approx \tilde{\tl}_L*\tl_L$ (a similar verification is illustrated in a previous version of this paper~\cite{Mar22}).

Finally, we prove that $G_{\D_L}^+$ is isomorphic to $\mathbb{Z} \oplus \mathbb{Z}$.
Given a tiling of $\D_L \times [0,2N]$ we obtain another tiling by reflecting on the plane $x+y=2L-1$ the dominoes in $(s_L \cup ([2L-2,2L] \times [0,1])) \times [0,2N]$; as with the two tilings in Figure~\ref{fig:irregz2}. 
This construction defines an automorphism $\psi \colon G_{\D_L}^+ \to G_{\D_L}^+$.
The map $\phi \colon G_{\D_L}^+ \to \mathbb{Z} \oplus \mathbb{Z}$ that takes a tiling $\tl$ to $(\textsc{Tw}(\tl), \textsc{Tw} \circ \psi (\tl))$ is then a homomorphism.
We have that $\phi(\tl_L) = (-1,1) $ and $\phi(\tilde{\tl}_L)=(1,1)$.
Thus, the image of $\phi$ is isomorphic to $\mathbb{Z} \oplus \mathbb{Z}$. 
Moreover, since $\tl_L$ and $\tilde{\tl}_L$ commute, $\phi$ is injective.\hfill $\diamond$
\end{example}

\section[Proof of Theorem 1]{Proof of Theorem~\ref{thm:sizecomp}}\label{sec:proofthmcomp}

Before proving Theorem~\ref{thm:sizecomp}, we outline the steps involved in the proof.
We begin by establishing a general result on random tilings of cylinders.
For a disk $\D$, we show that for any fixed family of tilings of $\R_{N_0}$, the probability that a random tiling of $\R_N$ contains few tilings from the fixed family decays exponentially as $N$ goes to infinity.
We then associate the probability that two (independent) random tilings of $\R_N$ are equivalent under flips with the probability of a random tiling, formed alternately by concatenations of tilings that are and are not contained in this fixed family, equals the vertical tiling.
Given a surjective homomorphism $\phi \colon G_{\D}^+ \to F_2$, we consider a fixed family of tilings corresponding to the values $a,a^{-1},b,b^{-1},e$.
By taking the image under $\phi$, we then relate the probability that two random tilings are equivalent under flips to the probability that a lazy random walk on $F_2$, which takes a deterministic step after each random step, returns to the identity. 
Theorem~\ref{thm:sizecomp} then follows from the fact that the later probability decays exponentially as the length of the walk increases.

Let $N_0 \in \mathbb{N}$ and consider a set of tilings $\mathcal{B} \subseteq \mathcal{T}(\R_{N_0})$.
We say that a tiling $\tl$ of a cork $\R_{N_1, N_2; p_1,p_2}$ is formed by $(k,M)$-blocks of $\mathcal{B}$ if there exist distinct nonnegative integers $b_1 < b_2 <\ldots < b_k$ and tilings $\bl_1, \bl_2, \ldots, \bl_k \in \mathcal{B}$ such that, for each $M_j = b_j(2M + N_0) + M+N_1$, the restriction of $\tl$ to $\D \times [M_j, M_j + N_0]$ equals~$\bl_j$.
We denote by $\text{block}_{\mathcal{B}}^{M}(\tl)$ the maximum nonnegative integer $k$ such that $\tl$ is formed by $(k,M)$-blocks of $\mathcal{B}$.
The following lemma shows that $\text{block}_{\mathcal{B}}^{M}(\tl)$ is almost never very small.
We provide an elementary proof of this result.
However, a more direct approach using large deviation principles for additive functionals of Markov chains might also be feasible (see, e.g., Chapter IV of~\cite{Hol2000}).

\begin{lemma}\label{lem:block}
Consider a quadriculated disk $\D$, $N_0 \in \mathbb{N}$ and $\mathcal{B} \subseteq \mathcal{T}(\R_{N_0})$.
For $N > N_0$, let $\mathbf{T}$ be a random tiling of $\D \times [0,N]$.
Then there exist $M \in 2\mathbb{N}$ and constants $C,c \in (0,1)$ such that $\mathbb{P}(\textup{block}_{\mathcal{B}}^{M}(\mathbf{T})<CN)=o(c^N)$.
\end{lemma}

\begin{proof}
Recall that $\pl_{\circ}$ is the empty plug.
If follows from from Lemma 13 of \cite{Sal21} that there exist $\epsilon > 0$ and an even integer $N_{\epsilon} > 2|\D|$ such that if $j \geq N_{\epsilon}$, $\tilde{N} \geq j +N_{\epsilon}$ and $p_0,p_{\tilde{N}} \in \mathcal{P}$ then for a random tiling $\mathbf{\tilde{T}}$ of $\mathcal{R}_{0,\tilde{N}; p_0,p_{\tilde{N}}}$ we have $\mathbb{P}(\textup{plug}_j(\mathbf{\tilde{T}})=\mathbf{p}_{\circ}) > \epsilon$.
Let $m=\underset{{p_i, p_j \in \mathcal{P}}}{\max}|\mathcal{T}(\R_{0, N_{\epsilon}; p_i, \pl_{\circ}})| |\mathcal{T}(\R_{0,N_{\epsilon}+N_0; \pl_{\circ}, p_j})|$ and $\delta = \epsilon |\mathcal{B}|m^{-1}$.
We claim that if $\tilde{N} = 2N_{\epsilon}+N_0$ then $\mathbb{P}(\text{block}_{\mathcal{B}}^{N_{\epsilon}}(\mathbf{\tilde{T}}) \geq 1) > \delta$.

Since $N_{\epsilon} \geq 2|\D|$, by Lemma $4.1$ of~\cite{Sal22}, there exist tilings $\tl_1$ of $\R_{0,N_{\epsilon},p_0, \pl_{\circ}}$ and $\tl_2$ of $\R_{0, N_{\epsilon},\pl_{\circ}, p_{\tilde{N}}}$. 
Thus, each tiling $\bl$ in $\mathcal{B}$ defines a tiling $ \tilde{\tl} =\tl_1 * \bl * \tl_2$ of the cork $\R_{0,\tilde{N};p_0,p_{\tilde{N}}}$ with $\textup{plug}_{N_{\epsilon}}(\mathbf{\tilde{\tl}})=\mathbf{p}_{\circ}$ and $\text{block}_{\mathcal{B}}^{N_{\epsilon}}(\mathbf{\tilde{\tl}}) =1$.
Moreover, notice that the number of tilings of $\R_{0,\tilde{N},p_0, p_{\tilde{N}}}$ such that the $N_{\epsilon}$-th plug equals $\pl_{\circ}$ is smaller than $m$.
Therefore, $\mathbb{P}(\text{block}_{\mathcal{B}}^{N_{\epsilon}}(\mathbf{\tilde{T}}) \geq 1 \mid \textup{plug}_{N_{\epsilon}}(\mathbf{\tilde{T}}) = \mathbf{p}_{\circ}) \geq |\mathcal{B}| m^{-1}$ and the claim above is proved.

Take $M= N_{\epsilon}$ and $C=\frac{\delta}{2(2M+N_0)}$.
For each $i=1,2, \ldots , \big\lfloor\frac{N}{2M+N_0}\big\rfloor$ let $A_i$ be the event that the restriction of $\mathbf{T}$ to $\D \times [(i-1)(2M+N_0), i(2M+N_0)]$ is formed by $(1,M)$-blocks of $\mathcal{B}$.
Notice that the previous paragraph implies that $\mathbb{P}(A_i \mid \mathbf{T} $\text{ constructed up to floor} $(i-1)(2M+N_0)) >\delta$.
Therefore, we have that $\mathbb{P}(\textup{block}_{\mathcal{B}}^{M}(\mathbf{T})<CN) \leq \mathbb{P}( X < CN)$ where $X$ is a random variable with binomial distribution $\text{Bin}(\big\lfloor \frac{N}{2M+N_0} \big\rfloor,\delta)$.
By writing $X$ as a sum of i.i.d.\ random variables with Bernoulli distribution Bern$(\delta)$, the result follows from Chernoff's inequality.
\end{proof}

\begin{lemma}\label{lem:ap}
Let $(X_t)_{t \geq 1}$ be a sequence of i.i.d.\ random variables in $F_2$ with $X_1$ uniformly distributed in $\{e,a,a^{-1},b,b^{-1}\}$.
Let $(y_t)_{t \geq 0}$ be a sequence of elements in $F_2$.
Then there exists $\alpha \in (0,1)$ such that $\prob(y_0X_1y_1\ldots X_ty_t=e) \leq \alpha^t$ for all $t \geq 0$.
\end{lemma}

\begin{proof}
We are interested in a random walk on $F_2$ where random and deterministic steps alternate, which is equivalent to applying a permutation to the vertices of $F_2$ after each random step.
By Theorem 1.1 of~\cite{ARSY24}, this process does not slow down the walk $(X_t)_{t \geq 1}$, i.e., $\prob(y_0X_1y_1\ldots X_ty_t=e) \leq \prob(X_1 \ldots X_t=e)$ for all $t \geq 0$; this inequality is also proved independently in a previous version of this paper~\cite{Mar22}.
In addition, by the Varopoulos-Carne bound (see, e.g., Theorem 13.4 of~\cite{LY17}), we obtain a constant $\alpha \in (0,1)$ such that $\prob(X_1\ldots X_t=e) \leq \alpha^t$ for all $t \geq 0$.
\end{proof}

\begin{proof}[Proof of Theorem~\ref{thm:sizecomp}]
Let $\phi \colon G_{\D}^+ \to F_2$ be a surjective homomorphism.
Consider $N_0 \in 2\mathbb{N}$ sufficiently large so that there exist tilings $\bl_a$ and $\bl_b$ of $\R_{N_0}$ with $\phi(\bl_a)=a$ and $\phi(\bl_b)=b$.
Let $\bl_{e}$ be the vertical tiling $\tl_{\text{vert},N_0}$.
Define the set of tilings $\mathcal{B} = \{ \bl_a, \bl_a^{-1}, \bl_b, \bl_b^{-1}, \bl_{e} \} \subset \mathcal{T}(\R_{N_0})$.
By Lemma~\ref{lem:block}, there exist $M \in 2\mathbb{N}$ and $C \in (0,1)$ such that the probability that $\text{block}_{\mathcal{B}}^{M}(\mathbf{T}_i)$ is smaller than $r= \lceil CN \rceil$ goes to zero exponentially. 
We may therefore assume that $\mathbf{T}_1$ and $\mathbf{T}_2$ are formed by at least $(r, M)$-blocks of $\mathcal{B}$.

Let $\mathcal{B}_r(\R_N) \subset \mathcal{T}(\R_N)$ be the set of tilings formed by at least $(r, M)$-blocks of~$\mathcal{B}$.
For each tiling $\tl \in \mathcal{B}_r(\R_N)$ let $b_{1,\tl} < b_{2,\tl} < \ldots < b_{r,\tl}$ be the first $r$ nonnegative integers such that, for $j=1,2, \ldots, r$ and $M_{j, \tl}=b_{j, \tl}(2M+N_0) + M$, the restriction of $\tl$ to $\D \times [M_{j, \tl}, M_{j, \tl}+N_0]$ equals a tiling in $\mathcal{B}$.
We now define an equivalence relation $\approxeq$ on $\mathcal{B}_r(\R_N)$: $\tilde{\tl} \approxeq \hat{\tl}$ if and only if $b_{j, \tilde{\tl}} = b_{j, \hat{\tl}}$ (for all $j$) and $\tilde{\tl}$ equals $\hat{\tl}$ in the region $(\D \times [0,N]) \smallsetminus (\bigcup_{j=1}^r \D \times [M_{j,\tilde{\tl}}, M_{j, \tilde{\tl}} + N_0])$.

Let $B_1, B_2, \ldots, B_l$ be the $\approxeq$-equivalence classes.
Notice that, for each $i \leq l$, there are fixed tilings $\tl_{0, i}, \tl_{1, i}, \ldots, \tl_{r, i}$ such that $B_i$ consists of all tilings of the form $\tl_{0,i}*\bl_1 * \tl_{1,i} * \bl_2 * \tl_{3,i} * \ldots * \bl_r * \tl_{r,i}$ with $\bl_1,\bl_2, \ldots, \bl_r \in \mathcal{B}$.
Thus, each $\approxeq$-equivalence class has size exactly $|\mathcal{B}|^r$.

Suppose that $\mathbf{T}_1$ has been chosen first from $\mathcal{B}_r(\R_N)$, say $\mathbf{T}_1 = \tl$.
The probability that there exists a sequence of flips joining $\mathbf{T}_2$ and $\tl$ is less than or equal to the probability that $\phi$ takes $\mathbf{T}_2*\tl^{-1}$ to the identity.
We prove that the later probability decays exponentially with $N$.
To this end, it suffices to show that the conditional probabilities $\prob(\phi(\mathbf{T}_2*\tl^{-1})=e \mid \mathbf{T}_2 \in B_i)$ are uniformly bounded by $\alpha^r$ for some constant $\alpha \in (0,1)$.

Consider a sequence $(X_t)_{t \geq 1}$ of i.i.d.\ random variables which assume values in $F_2$ such that $\mathbb{P}(X_1=x) = \frac{1}{5}$ for $x \in \{a,a^{-1},b,b^{-1},e\}$.
Then, by construction, $\prob(\phi(\mathbf{T}_2*\tl^{-1})=e \mid \mathbf{T}_2 \in B_i)$ is equal to  $\prob(\phi(\tl_{0, i})X_1\phi(\tl_{1, i}) \ldots X_{r} \phi(\tl_{r,i}*\tl^{-1}) =e)$. 
The result now follows from Lemma~\ref{lem:ap}.
\end{proof}

\section{The domino group of thin rectangles}\label{sec:thinrec}

\vbox{In this section we prove Theorem~\ref{thm:thinrectangles}, that is, we compute the even domino group of $\D_L = [0,L] \times [0,2]$ for $L \geq 3$.
The strategy consists in constructing a homomorphism from $G_{\D_L}^+$ to the group $G_L^+$ defined in Equation~\ref{eq:1.1}.
We then prove that this homomorphism is in fact an isomorphism.}

The computation of $G_{\D_L}^+$ is inspired by the results of \cite{MS15}, where a flip invariant for tilings of duplex regions is exhibited. 
By performing a rotation, we think of tilings of $\D_L \times [0,N]$ as tilings of the duplex region $[0,N] \times [0,L] \times [0,2].$
Therefore, in this section, we say that a flip is horizontal if it is performed in two dominoes contained in one of the two floors of the new rotated tiling; the flip is vertical otherwise.  

Let $\tl$ be a tiling of $\D_L \times [0,N]$ and orient each domino contained in $\tl$ from its white unit cube to its black unit cube.
By projecting the two floors of $\tl$ on the plane $z=0$, we obtain a \emph{diagram} $\mathcal{I}_{\tl}$ on $[0,N] \times [0,L]$ containing oriented disjoint cycles and \emph{jewels}, i.e., unit squares formed by the projections of dominoes parallel to the $z$-axis.
A cycle is \emph{trivial} if it has length two and a jewel is \emph{trivial} if it is not enclosed by a cycle.
The color of a jewel $j$ is defined as the color of its corresponding unit square in the rectangle $[0,N] \times [0,L]$.
We write $\text{color}(j)=+1$ if $j$ is white and $\text{color}(j)=-1$ if $j$ is black.
The Figure~\ref{fig:duplex} shows an example of a tiling and its associated diagram; we always exhibit trivial cycles in green, counterclockwise cycles in red and clockwise cycles in blue.
\begin{figure}[H] 
\centerline{
\includegraphics[width=0.65\textwidth]{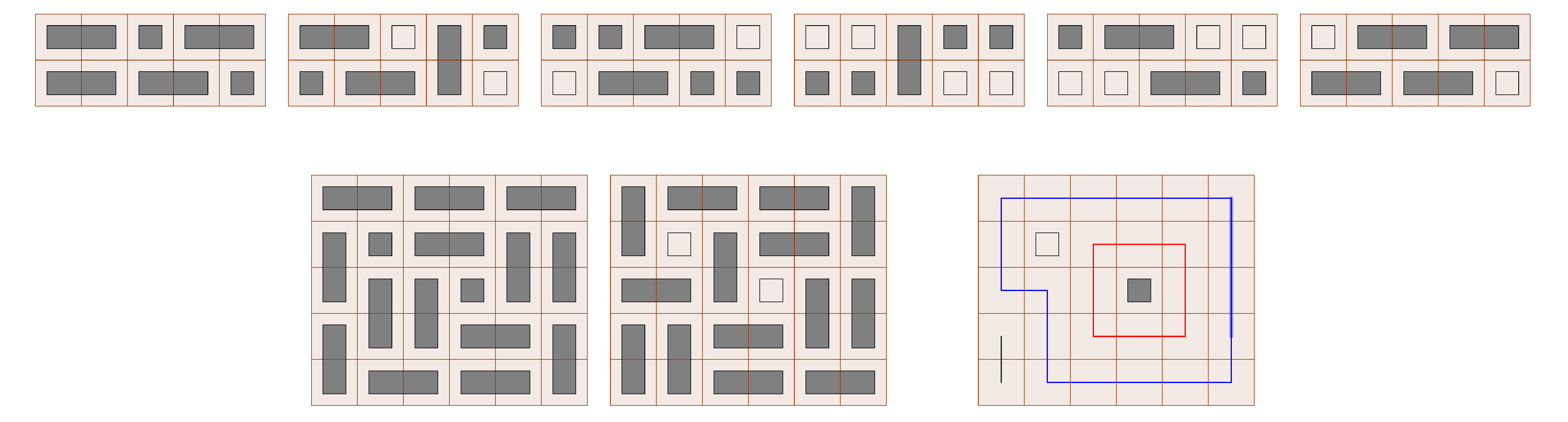}}
\caption{The first row shows a tiling $\tl$ of $\D_5 \times [0,6]$.
The second row shows $\tl$ after a rotation and its diagram $\mathcal{I}_{\tl}$.}
\label{fig:duplex}
\end{figure}

We order the jewels contained in $\mathcal{I}_{\tl}$.
Consider two jewels $j_1 = [a,a+1] \times [b,b+1]$ and $j_2 = [c,c+1] \times [d,d+1]$.
If $j_1$ and $j_2$ are in different columns, i.e. $a \neq c$, we write $j_1 < j_2$ if $a<c$.
If $j_1$ and $j_2$ are in the same column, i.e. $a=c$, we write $j_1 < j_2$ if $b > d$.
When this order is used, jewels are called \textit{ordered jewels}.

For a jewel $j$ let $\text{wind}(j)$ be the sum of the winding numbers $\text{wind}(j, \gamma)$ taken over all the cycles $\gamma$ in $\mathcal{I}_{\tl}$.
Notice that $\text{wind}(j)$ is an integer and $|\text{wind}(j)| \leq \genfrac{\lfloor}{\rfloor}{0.5pt}{1}{L-1}{2}$.
We are especially interested in the winding numbers of jewels contained in the same column.
Consider the finite set $R_L = \{ (m,n) \in \mathbb{Z}^2 \colon \max\{|m|,|n|,|m-n|\} < \genfrac{\lfloor}{\rfloor}{0.5pt}{1}{L}{2} \}$.
The Figure~\ref{fig:RlZ2} below shows the elements of $R_{10}$ and two tilings of $\D_{10} \times [0,12]$, notice that in the figure any two jewels $j_1$ and $j_2$ contained in the same column are such that $(\win(j_1), \win(j_2)) \in R_{10}$.
\begin{figure}[H] 
\centerline{
\includegraphics[width=0.75\textwidth]{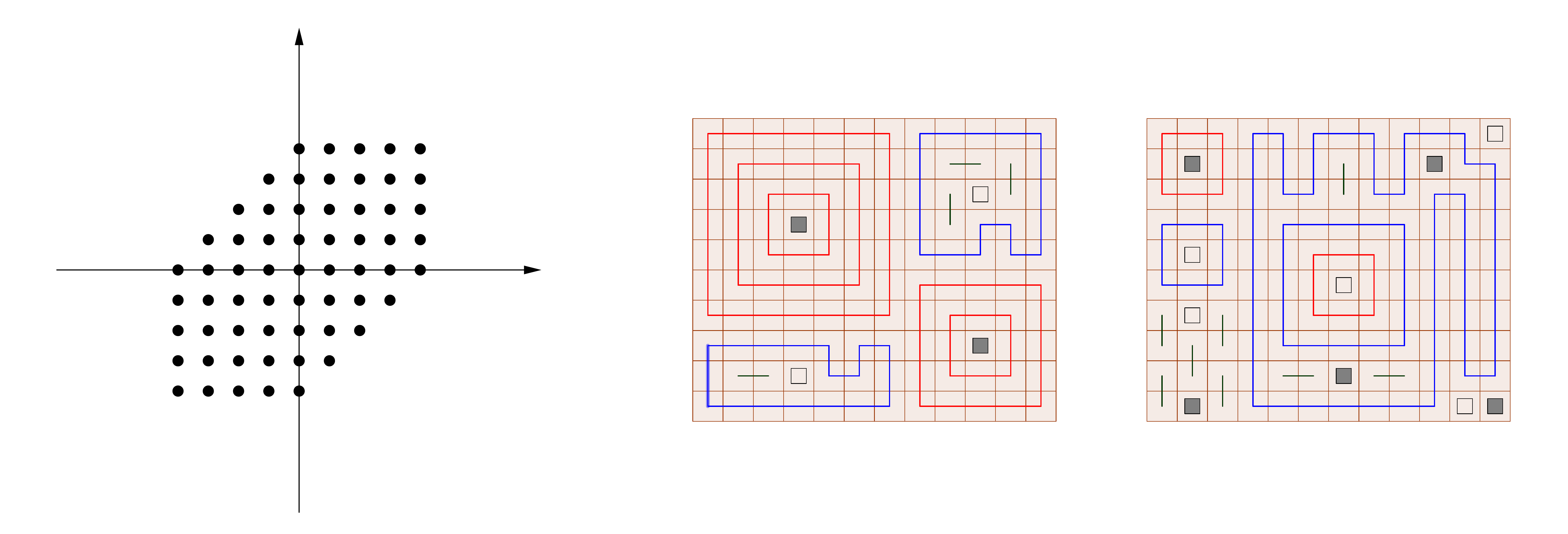}}
\caption{The cartesian plane with elements of $R_{10}$ shown in black, and two diagrams of tilings of $\D_{10} \times[0,12].$}
\label{fig:RlZ2}
\end{figure}

\begin{lemma}\label{lem:vertflip} 
Consider $L \geq 3$.
Then $(m,n) \in R_L$ if and only if there exist $N \in \mathbb{N}$ and a tiling $\tl$ of $\D_L \times [0,2N]$ such that $\mathcal{I}_{\tl}$ contains two jewels $j_1$ and $j_2$ in the same column with $(\win(j_1), \win(j_2))=(m,n).$
\end{lemma}

\begin{proof}
We first prove the \emph{if} direction.
Suppose that $j_1=[a,a+1] \times[b,b+1]$ and $j_2=[a,a+1] \times [c,c+1]$ with $b>c$.
Since $\win(j_1) = m$ there exists at least $|m|$ unit squares in $[a,a+1] \times [b+1,L].$
Analogously, there exists at least $|n|$ unit squares in $[a,a+1] \times [0,c].$	
Moreover, we must have at least $|m-n|$ cycles not enclosing both jewels. Therefore, there exists at least $|m-n|$ unit squares in $[a,a+1] \times  [b+1,c]$.
Then, $|m|+|n|+|m-n| \leq L-2$ and we have $ \max\{|m|,|n|,|m-n|\} \leq \frac{L-2}{2}$.

For the \emph{only if} direction let $(m,n) \in R_L$ and take $N$ sufficiently large.
In order to show the existence of a tiling $\tl$ of $\D_L \times [0,2N]$ with the desired properties we proceed backwards.
Indeed, since a tiling is entirely determined by its diagram, it suffices to construct $\mathcal{I}_{\tl}$.

We first deal with the case in which $|m|+|n| < \lfloor \frac{L}{2} \rfloor$.
Consider two disjoint squares centered in the same column: $s_m$ of side $2|m|+1$ and $s_n$ of side $2|n|+1$.
Let the jewel $j_1$ (resp. $j_2$) be the center of $s_m$ (resp. $s_n$).
Construct $|m|$ cycles in $s_m$ and $|n|$ cycles in $s_n$ such that $\win(j_1) = m$ and $\win(j_2) =n$.
Now, to obtain $\mathcal{I}_{\tl}$, fill the rest of $[0,2N] \times [0,L]$ with trivial cycles and trivial jewels.

We are left with the case $|m|+|n| \geq \lfloor \frac{L}{2} \rfloor$, so that $\sign(m)=\sign(n)$.
Suppose that $|m| \geq |n|$ and write $|n|=|m| - r$, where $0 \leq r \leq |m|.$
The square $s=[0,2|m|+2]^2$ is contained in $[0,2N] \times [0,L]$, since $|m| < \lfloor \frac{L}{2} \rfloor$.
Let $j_1=[|m|,|m|+1] \times [|m|+1, |m|+2]$ and $j_2=[|m|,|m|+1] \times [|m|-r,|m|-r+1]$.
Construct $m$ cycles in $s$ such that $\win(j_1)=m$.
We have $\win(j_2)=n$, as $\sign(m)=\sign(n)$.
The result then follows by proceeding as in the previous paragraph.
\end{proof}

Recall that, as in Equation~\ref{eq:1.1}, $G_L^+ = \langle S_L \mid [a_m, a_n] =1 \text{ for } (m,n) \in R_L \rangle$, where $S_L= \{ a_i \colon i \in \mathbb{Z}_{\neq 0}\text{  and } |i| \leq \genfrac{\lfloor}{\rfloor}{0.5pt}{1}{L-1}{2} \}$.
We construct a map $$\Phi \colon \underset{N \geq 1}{\bigcup} \mathcal{T}(\D_L \times [0,2N]) \to G_L^+.$$
Consider a tiling $\tl$ of $\D_L \times [0,2N]$ and let $j_1 < j_2 < \ldots < j_k$ be the ordered jewels in $\mathcal{I}_{\tl}$. 
Define $\Phi(\tl)=b_1 \ldots b_k$ where $b_i = a_{\win(j_i)}^{\col(j_i)}$ if $\win(j_i) \neq 0$ and $b_i = e$ if $\win(j_i) = 0$. 

\begin{lemma}\label{lem:homthinrect}
The map $\Phi$ induces a homomorphism $\phi \colon G_{\D_L}^+ \to G_L^+$.
\end{lemma}

\begin{proof}
It suffices to check that $\Phi$ is invariant under flips.
Let $\tl$ be an arbitrary tiling of $\D_L \times [0,2N]$.
Consider a horizontal flip performed in two dominoes $d_1$ and $d_2$.
The horizontal flip either connects two disjoint cycles or disconnects a cycle into two cycles.
Suppose the former, the other case is similar.
Therefore, $d_1$ and $d_2$ are contained in distinct cycles.
If either $d_1$ or $d_2$ is contained in a trivial cycle then it is easy to see that the horizontal flip does not change the winding number of any jewel.
Then, suppose that $d_1$ and $d_2$ are contained in nontrivial cycles, Figure~\ref{fig:flipcycle} below shows an example of the possible cases.
\begin{figure}[H] 
\centerline{
\includegraphics[width=0.55\textwidth]{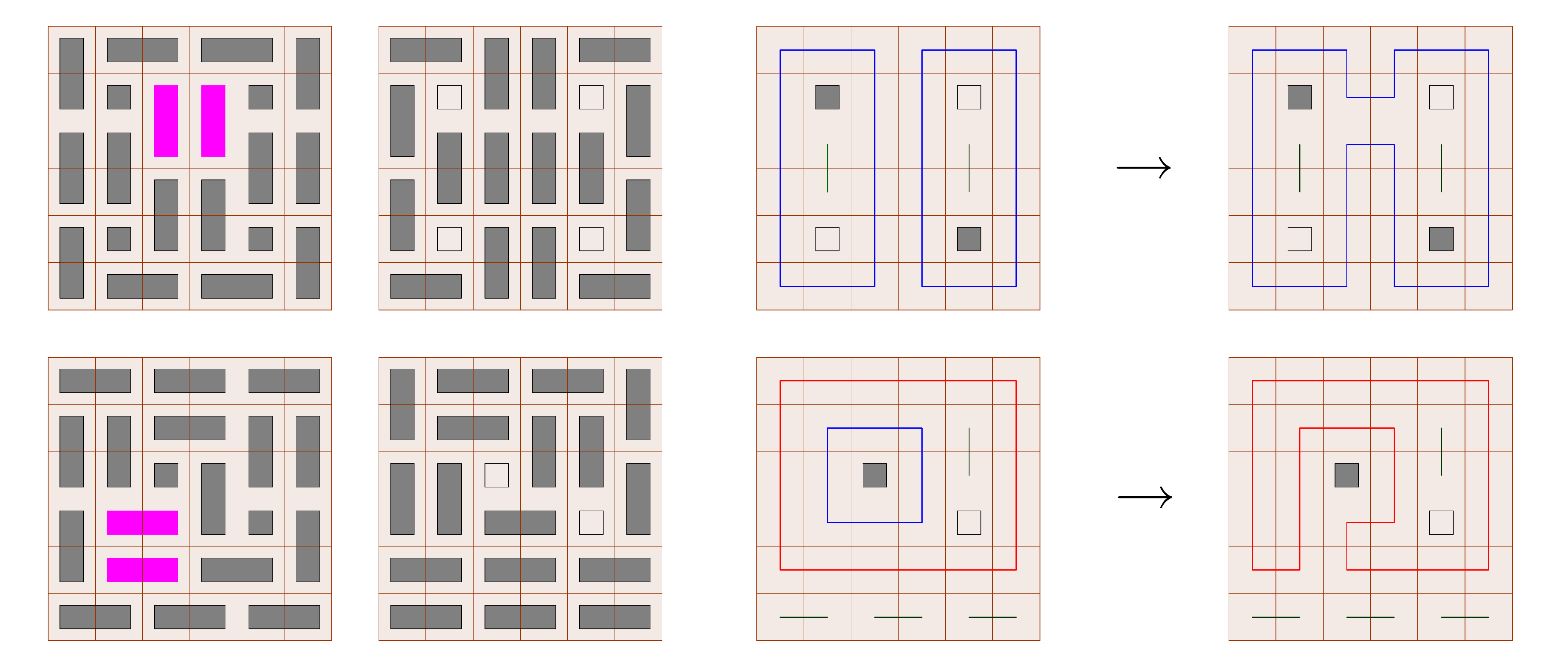}}
\caption{Two tilings and the effect of a horizontal flip (highlighted in magenta) on their diagrams.}
\label{fig:flipcycle}
\end{figure}
If $d_1$ and $d_2$ are contained in cycles having the same orientation then the flip connects the two cycles preserving the orientation.
If $d_1$ and $d_2$ are contained in cycles having opposite orientations then one cycle must be enclosed by the other.
The flip then creates a cycle with the same orientation as the outer cycle.
Moreover, the new cycle encloses only jewels enclosed by the outer cycle but not by the inner cycle.
Therefore, in any of the possible cases, the flip preserves the winding number of the jewels.
Then, $\Phi$ is invariant under horizontal flips.

Consider a vertical flip that takes a trivial cycle to two adjacent jewels (i.e., jewels whose corresponding unit squares are adjacent); a similar analysis holds for the reverse of this flip.
The flip creates adjacent jewels $j$ and $j'$ such that $j<j'$, $\win(j) = \win(j')$ and $\col(j) \neq \col(j')$.
If $j$ and $j'$ are in the same column then the flip clearly preserves the value of $\Phi(\tl)$.
Otherwise, the definition of $G_L^+$ and Lemma~\ref{lem:vertflip} imply that the contributions of jewels between $j$ and $j'$ commute so that $\Phi$ is invariant under vertical flips.
\end{proof}

Our objective is to prove that the homomorphism $\phi$ obtained above is an isomorphism.
To achieve this, we now study the even domino group $G_{\D_L}^+$.
We follow~\cite{MS15} to derive a family of generators of $G_{\D_L}^+$. 
A tiling $\tl$ of $\D_L \times [0,N]$ is called a \emph{boxed tiling} if its corresponding diagram $\mathcal{I}_{\tl}$ is composed of a nontrivial jewel $j$ and trivial jewels outside the square of center $j$ and side $2|\win(j)|+1$.

We prefer to work with boxed tilings due to some helpful properties.
Notably, we can move via flips the nontrivial jewel of a boxed tiling so that the resulting tiling is a boxed tiling as well. 
Specifically, consider two boxed tilings $\tl$ and $\tilde{\tl}$ of $\D_L \times [0,N]$.
Let $j$ and $\tilde{j}$ be the nontrivial jewels of $\mathcal{I}_{\tl}$ and $\mathcal{I}_{\tilde{\tl}}$, respectively.
If $\col(j)=\col(\tilde{j})$ and $\win(j) = \win(\tilde{j})$ then $\tl \approx \tilde{\tl}$.
For instance, Figure~\ref{fig:boxedjewelsflips} shows the process of moving a nontrivial jewel with winding number equals 1.
The extension to other cases follows inductively, by initially transforming the outer cycle into a rectangle through a sequence of flips.

\begin{figure}[H] 
\centerline{
\includegraphics[width=0.927\textwidth]{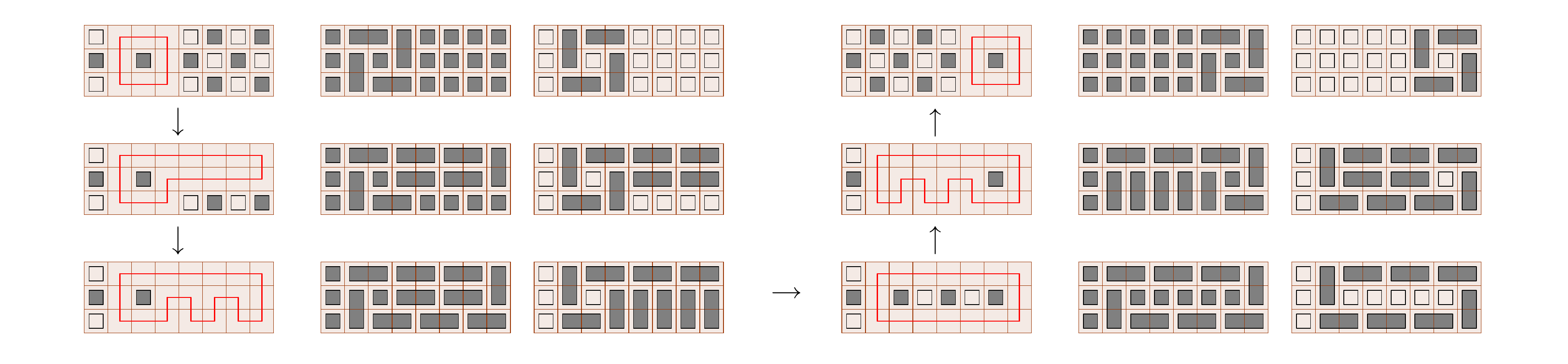}}
\caption{The process of moving a nontrivial jewel via a sequence of flips.}
\label{fig:boxedjewelsflips}
\end{figure}

The family of boxed tilings generates the even domino group $G_{\D_L}^+$.
Indeed, every tiling of a cylinder $\D_L \times [0,2N]$ is $\sim$-equivalent to a concatenation of boxed tilings.
\begin{lemma}\label{lem:sepjew}
Let $\tl$ be a tiling of $\D_L \times [0,2N]$.
Then, there exist $M \in \mathbb{N}$ and boxed tilings $\tl_1, \ldots, \tl_k$ of $\D_L \times [0,2M]$ such that $\tl \sim \tl_1*\ldots*\tl_k$.
\end{lemma}

\begin{proof}
This result is proved, in Lemma 7.4 of~\cite{MS15}, for diagrams in $\mathbb{Z}^2$ instead of $[0,2N]\times[0,L]$.
However, the same proof holds in our setting, since the relation $\sim$ allows us to assume that $N$ is arbitrarily large.
To avoid repetition, we do not provide further details.
\end{proof}

We now investigate relations between boxed tilings.
The lemma below shows that, under specific conditions, two boxed tilings commute with respect to concatenation.
Notice that the particular case, where the nontrivial jewels of the two boxed tilings share the same color and winding number, follows from the fact that we can move nontrivial jewels.

\begin{lemma}\label{lem:comjew}
Consider boxed tilings $\tl_1$ of $\D_L \times [0,2N_1]$ and $\tl_2$ of $\D_L \times [0,2N_2]$.
Let $j_1$ and $j_2$ be the nontrivial jewels in $\mathcal{I}_{\tl_1}$ and $\mathcal{I}_{\tl_2}$, respectively.
If $\col(j_1)=\col(j_2)$ and $(\win(j_1),\win(j_2)) \in R_L$ then $\tl_1 * \tl_2 \approx \tl_2* \tl_1$.
\end{lemma}

\begin{proof}
Let $(m,n)=(\win(j_1),\win(j_2))$, we focus on the diagram $\mathcal{I}_{\tl_1 * \tl_2}$.
We consider two cases: $|m|+|n| < \lfloor \frac{L}{2} \rfloor$ and $|m|+|n| \geq \lfloor \frac{L}{2} \rfloor$.
First suppose the former.
This case is a matter of moving the nontrivial jewels (as in Figure~\ref{fig:boxedjewelsflips}), we proceed in three steps.
Initially, move $j_1$ to a jewel $\widetilde{j_1}$ in $[0,2N_1] \times [2|n|+1, L]$ and $j_2$ to a jewel $\widetilde{j_2}$ in $[2N_1,2(N_1+N_2)] \times [0,2|n|+1]$.
Since $L \geq 2(|m|+ |n| +1)$, we can then move $\widetilde{j_1}$ to a jewel $\overbar{j_1}$ in $[2N_2,2(N_1+N_2)] \times [2|n|+1, L]$ and $\widetilde{j_2}$ to a jewel $\overbar{j_2}$ in $[0,2N_2] \times [0,2|n|+1]$.
Finally, move $\overbar{j_1}$ (resp. $\overbar{j_2}$) in $[2N_2,2(N_1+N_2)] \times [0,L]$ (resp. $[0,2N_2] \times [0,L]$) to obtain copies of $\tl_1$ and $\tl_2$.
For instance, Figure~\ref{fig:comtil} shows a particular case ($L=6$ and $(m,n) = (1,-1)$) of the general idea.

\begin{figure}[H] 
\centerline{
\includegraphics[width=0.68\textwidth]{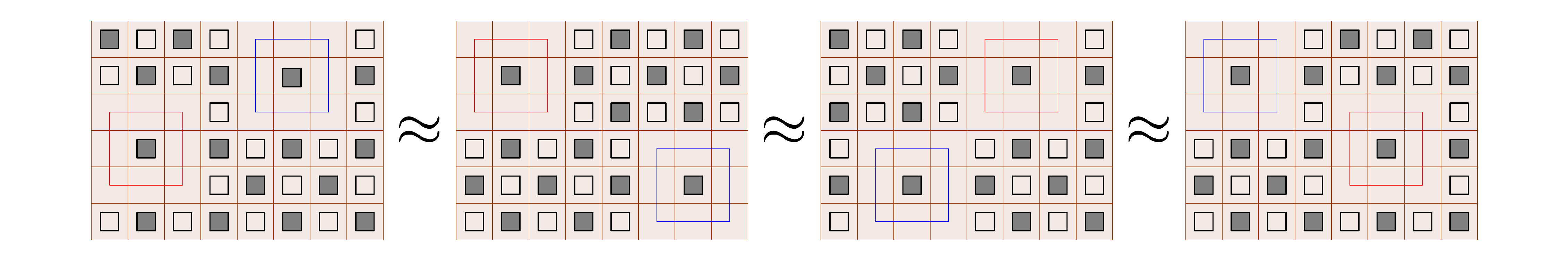}}
\caption{The diagram $\mathcal{I}_{\tl_1*\tl_2}$ and the effect of three sequences of flips.}
\label{fig:comtil}
\end{figure}

We have $R_3= \{ (0,0)\}$ and $R_4 = R_5= \{(0,0),(\pm 1,0), (0, \pm 1), (1,1), (-1,-1) \}$.
The previous paragraph cover all cases except $(m,n) \in \{(1,1), (-1,-1) \}$ for $L=4$ and $L=5$.
However, in any of these two cases, the nontrivial jewels share the same color and winding number, so that $\tl_1 \approx \tl_2$.
Thus, the result holds for $L=3,4,5$.

If $|m|+|n| \geq \lfloor \frac{L}{2} \rfloor$ then $\sign(m)=\sign(n)$.
The proof follows by induction on~$L$.
Perform a sequence of flips that takes the largest cycle which encloses $j_1$ to the cycle $\gamma_1$ which encloses the region $[1,2N_1-1] \times [1,L-1]$.
Similarly, enlarge the largest cycle which encloses $j_2$ to obtain the cycle $\gamma_2$ which encloses the region $[2N_1+1,2N_1+2N_2-1] \times [1,L-1]$.
Since $\sign(m)=\sign(n)$ there exists a flip that connects $\gamma_1$ and $\gamma_2$ into a cycle $\gamma$; as before, enlarge $\gamma$ to obtain a cycle which encloses the region $[1,2N_1+2N_2-1] \times [1,L-1]$.

Now, the winding number of $j_1$ and $j_2$, restricted to $[1,2N_1+2N_2-1] \times [1,L-1]$, is equal to $m-\sign(m)$ and $n-\sign(n)$, respectively.
Notice that $(m-\sign(m), n-\sign(n)) \in R_{L-2}$.
Then, by the induction hypothesis, there exists a sequence of flips which commutes the nontrivial jewels in $[1,2N_1+2N_2-1] \times [1,L-1]$.
Finally, undo the flips of the previous paragraph to conclude that $\tl_1 *\tl_2 \approx \tl_2*\tl_1$. 
\end{proof}

\begin{proof}[Proof of Theorem~\ref{thm:thinrectangles}]
We prove that the homomorphism $\phi \colon G_{\D_L}^+ \to G_L^+$ of Lemma~\ref{lem:homthinrect} is an isomorphism.
Indeed, we obtain a map $\psi: G_L^+ \to G_{\D_L}^+$ such that $\psi^{-1} = \phi$.
To this end, we first define a homomorphism $\Psi: F(S_L) \to G_{\D_L}^+$ from the free group generated by $S_L$ to $G_{\D_L}^+$.

Consider a  nonzero integer $|i| \leq \genfrac{\lfloor}{\rfloor}{0.5pt}{1}{L-1}{2}$.
Let $\tl_i$ be the boxed tiling, of the cylinder $\D_L \times [0,2(|i|+1)]$, such that the nontrivial jewel $j_i$ in $\mathcal{I}_{\tl_i}$ is the center of the square $[0,2|i|+1]^2$ and $\win(j_i)=i$.
Let $\Psi$ be the homomorphism such that $\Psi(a_i) = \tl_i$.
Then, it follows from Lemma~\ref{lem:comjew} that  $\Psi(a_ma_na_m^{-1}a_n^{-1}) = e$ for $(m,n) \in R_L$.
Thus, $\Psi$ induces a homomorphism $\psi \colon G_L^+ \to G_{\D_L}^+$.

By definition $\phi(\tl_i)= a_i$, so that $\phi \circ \psi$ equals the identity map.
Now, consider an arbitrary tiling $\tl$ of $\D_L \times [0,2N]$.
By Lemma~\ref{lem:sepjew}, $\tl$ is $\sim$-equivalent to a concatenation of boxed tilings.
Moreover, we know that two boxed tilings whose nontrivial jewels share the same color and winding number are also $\sim$-equivalent.
Thus, there exists $i_1,i_2, \ldots, i_k \in \mathbb{Z}$ such that $\tl \sim \tl_{i_1}^{\epsilon_1} * \tl_{i_2}^{\epsilon_2} * \ldots * \tl_{i_k}^{\epsilon_k}$ for some $\epsilon_1, \epsilon_2, \ldots, \epsilon_k = \pm 1$.
Then, $\phi(\tl) = \phi(\tl_{i_1}^{\epsilon_1} * \tl_{i_2}^{\epsilon_2} * \ldots * \tl_{i_k}^{\epsilon_k})=a_{i_1}^{\epsilon_1}a_{i_2}^{\epsilon_2} \ldots a_{i_k}^{\epsilon_k}$, and therefore $\psi \circ \phi (\tl)= \tl$.
\end{proof}

\begin{remark}\label{rem:semidirectprod}
We are able to compute the domino group $G_{\D_L}$ once we compute the even domino group $G_{\D_L}^+$.
Indeed, consider a tiling of $\D_L \times [0,1]$ so that it defines an element of order 2 in $G_{\D_L}$ which generates the subgroup $H$.
Notice that every element in $G_{\D_L}$ is a product of an element in $G_{\D_L}^+$ and an element in $H$.
Thus, the domino group is isomorphic to the inner semidirect product of $G_{\D_L}^+$ and $H$.
More specifically, let $\psi \colon \mathbb{Z} / (2) \to \text{Aut}(G_L^+)$ be the homomorphism defined by $\psi(1)(a_i)=a_{-i}^{-1}$ for each $a_i \in S_L$.
Therefore, the semidirect product $\mathbb{Z}/(2) \ltimes_{\psi} G_L^+$ is isomorphic to the domino group $G_{\D_L}$. \hfill $\diamond$
\end{remark}

\section{Strongly irregular disks}\label{sec:proofthms123}

In this section, we provide a detailed study of strongly irregular disks, including the proof of Theorem~\ref{thm:irregdisk-squrdisc}.
We also present Theorems~\ref{thm:irregdisk-dominodisc} and~\ref{thm:dominodisc2}, which demonstrate the strong irregularity of certain disks that can be disconnected by removing a domino, such as the disks shown in Figures~\ref{fig:irregdisk-dominodisc} and~\ref{fig:examthm2dominodisc}.
The proofs of these theorems closely follow the proof of Theorem~\ref{thm:irregdisk-squrdisc} and are included to ensure the completeness of the text.
Notice that the hypotheses of Theorems~\ref{thm:irregdisk-squrdisc}, \ref{thm:irregdisk-dominodisc} and \ref{thm:dominodisc2} are not mutually exclusive.
For instance, both first and third disks in Figure~\ref{fig:irregdisk-dominodisc} are examples where all three theorems apply.
However, in such cases, each theorem provides a distinct surjective homomorphism.  

\begin{theorem}\label{thm:irregdisk-dominodisc}
Consider a balanced quadriculated disk $\D$ containing a domino $d$ such that $\D \smallsetminus d$ is not connected.
Suppose that there exists a $2 \times 2$ square in $\D$ which contains $d$.
If every connected component of $\D \smallsetminus d$ which intersects a $2 \times 2$ square contained in $\D$ that contains $d$ has size at most $\frac{|\D|-2}{2}$ then $\D$ is strongly irregular.
\end{theorem}
\begin{figure}[H]
\centerline{
\includegraphics[width=0.61\textwidth]{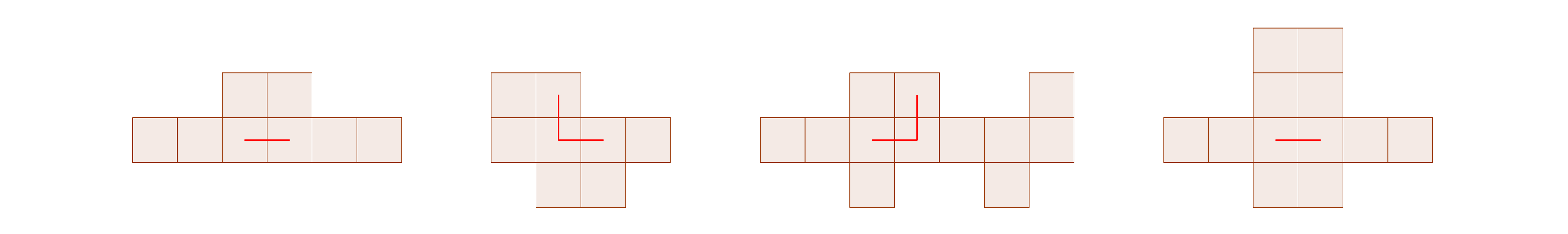}}
\caption{Examples of strongly irregular disks; dominoes $d$ as in Theorem~\ref{thm:irregdisk-dominodisc} are marked by a red line segment.}
\label{fig:irregdisk-dominodisc}
\end{figure}

\begin{theorem}\label{thm:dominodisc2} Consider a balanced quadriculated disk $\mathcal{D}$.
Suppose there exists a $2 \times 2$ square $s \subset \mathcal{D}$ such that $\mathcal{D} \smallsetminus s$ is the union of two disjoint disks $\mathcal{D}_1$ and $\mathcal{D}_2$ with $| \mathcal{D}_1|=|\mathcal{D}_2|$.
Suppose $s$ contains dominoes $d_1$ adjacent to $\D_1$ and $d_2$ adjacent to $\D_2$ such that $\D \smallsetminus d_1$ and $\D \smallsetminus d_2$ are not connected.
Then, $\D$ is strongly irregular.
\end{theorem}

\begin{figure}[H]
\centerline{
\includegraphics[width=0.41\textwidth]{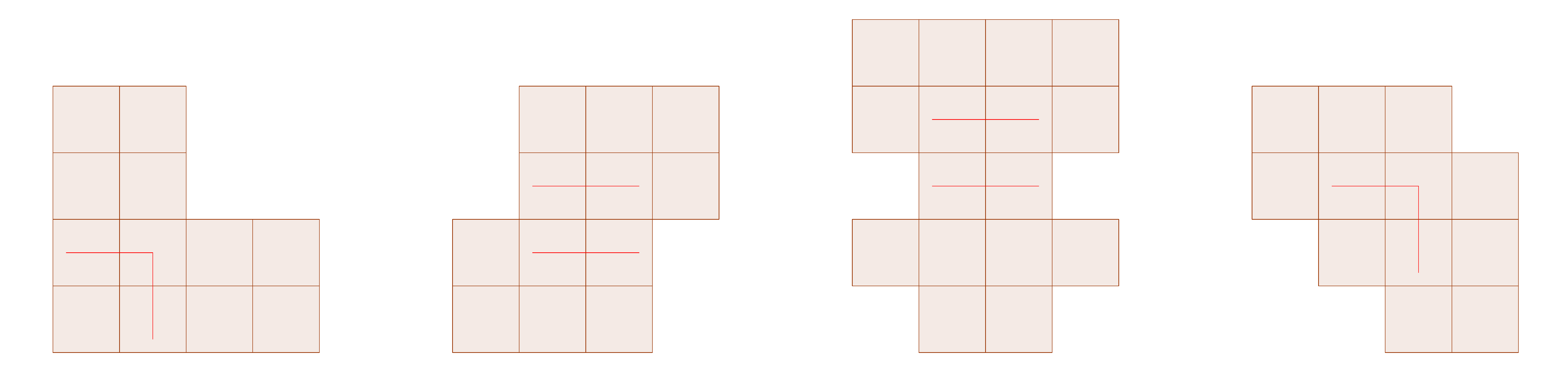}}
\caption{Examples of strongly irregular disks; dominoes $d_1$ and $d_2$ as in Theorem~\ref{thm:dominodisc2} are marked by a red line segment.}
\label{fig:examthm2dominodisc}
\end{figure}

\begin{remark}
In Theorem~\ref{thm:irregdisk-dominodisc}, the hypothesis about the size of the connected components cannot be discarded.
Indeed, consider an even number $a \geq 4$ and let $b=a^2-4$.
Then, it follows from~\cite{Mar21}, that the disk $\D$ formed by the union of $[0,a]^2$ and $[a,a+\frac{b}{2}] \times [0,2]$ is regular.
For instance, Figure~\ref{fig:regcounterexample} shows $\D$ for $a=4$.
On the other hand, notice that if $d=[a-1,a] \times [0,2]$ then $\D \smallsetminus d$ is the union of two disjoint disks $\D_1$ and $\D_2$ such that $|\D_1|=a^2-2$ and $|\D_2|=b=a^2-4$. \hfill $\diamond$
\end{remark}

In order to prove the strongly irregularity of a disk, we construct a surjective homomorphism from its even domino group to the free group $F_2$.
The homomorphisms presented in this section are all constructed based on the same idea.
Given a disk $\D$, we first construct a map $\Phi$ that takes oriented edges (floors with parity) in $\mathcal{C}_\D^+$ to $F_2$.
Thus, $\Phi$ defines a homomorphism from the free group on oriented edges to $F_2$.
We then verify that $\Phi$ maps the boundary of any 2-cell to the identity.
Consequently, we obtain a homomorphism $\phi \colon G_\D^+ \to F_2$ by taking the quotient of the free group on oriented edges by the relations defining $G_{\D}^+$.
In order to fix the ideas, consider Example~\ref{exm:diskf2}, which exhibits the only disk that we know of whose strong irregularity does not follow from Theorems~\ref{thm:irregdisk-squrdisc}, \ref{thm:irregdisk-dominodisc} and~\ref{thm:dominodisc2}.

\begin{example}\label{exm:diskf2}
Let $\D$ be the first disk in the bottom row of Figure~\ref{fig:nonreg}, i.e., the quadriculated disk formed by the union of the rectangle $[0,4] \times [0,1]$ and the two unit squares $[1,2] \times [-1,0]$ and $[2,3] \times [1,2]$.
We construct a surjective homomorphism $\phi \colon G_{\D}^+ \to F_2$.

Let $d=[1,3] \times [0,1]$ be a domino.
Consider the plugs $p_0 = [0,1]^2 \cup ([2,3] \times [1,2])$ and $p_1 = ([1,2]\times[-1,0]) \cup ([3,4]\times [0,1])$.
Then, $f=(p_0, \{d\}, p_1)$ is an oriented edge of the complex $\mathcal{C}_\D$.
This edge defines four oriented edges in $\mathcal{C}_{\D}^+$, i.e., floors with parity: $\mathbf{f_0} = (f,0)$, $\mathbf{f_1} = (f^{-1},0)$, $\mathbf{f_0}^{-1} = (f^{-1},1)$ and $\mathbf{f_1}^{-1} = (f,1)$.

We now define $\Phi$ for oriented edges in $\mathcal{C}_{\D}^+$.
Set $\Phi(\mathbf{f_0})=a$, $\Phi(\mathbf{f_1}) =b$, $\Phi(\mathbf{f_0}^{-1}) =a^{-1}$ and $\Phi(\mathbf{f_1}^{-1}) =b^{-1}$; all other edges are mapped to the identity.
Notice that $\Phi$ takes the boundary of any 2-cell in $\mathcal{C}_{\D}^+$ to the identity, as $\mathbf{f_i}$ (for $i=0,1$) do not permit horizontal flips and the possible vertical flip is obtained by moving consecutively along $\mathbf{f_i}$ and $\mathbf{f_i}^{-1}$.
Thus, $\Phi$ extends to a homomorphism $\phi \colon G_{\D}^+ \to F_2$.
The tilings $\tl$ and $\tilde{\tl}$ shown in Figure~\ref{fig:generatorsf2} are such that $\phi(\tl)=a$ and $\phi(\tilde{\tl})=b$, so that $\phi$ is surjective. 
By following the ideas of Example~\ref{exm:diskz2}, one can further show that $\tl$ and $\tilde{\tl}$ generate $G_\D^+$, implying that $\phi$ is in fact an isomorphism.$\hfill \diamond$
\end{example}

\begin{figure}[H]
\centerline{
\includegraphics[width=1.61\textwidth]{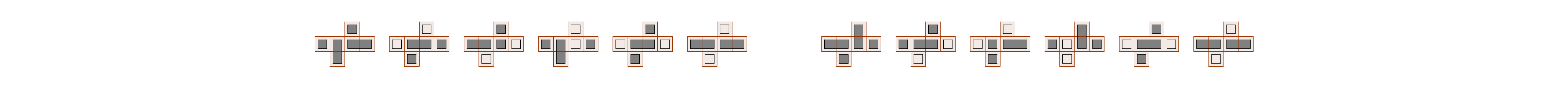}}
\caption{Tilings $\tl$ and $\tilde{\tl}$ of $\D \times [0,6]$.
The map $\Phi$ takes the second floor of $\tl$ and $\tilde{\tl}$ to $a$ and $b$, respectively.}
\label{fig:generatorsf2}
\end{figure}

In general, the proof of the surjectivity of $\phi$ is based on an algorithm to construct, for each plug $p \in \mathcal{P}$, a tiling $\tl_p$ of $\R_{0,|p|;p, \pl_{\circ}}$.
We now describe this algorithm (see Lemma 4.2 of \cite{Sal22}).
We first identify a disk $\D$ with a bipartite graph $\mathcal{G}(\D)$ whose vertices are the white and black colored unit squares in $\D$; two vertices are connected by an edge if and only if their corresponding unit squares are adjacent.
We refer to a spanning tree of $\mathcal{G}(\D)$ as a spanning tree of $\D$.
In that sense, given a disk $\D$ with a spanning tree $T$, we say that a domino in $\D$ is not an edge of $T$ if its corresponding edge in $\mathcal{G}(\D)$ is not contained in $T$.

Let $\D$ be a disk with a spanning tree $T$ and a plug $p \in \mathcal{P}$.
For any two unit squares $s$ and $\tilde{s}$ in $\D$ define their distance $d(s, \tilde{s})$ as the length of the path in $T$ connecting them.
Consider a sequence of plugs $p_0, p_1, \ldots, p_{\frac{|p|}{2}}$ such that $p_0 = p$ and $p_{i+1}$ is obtained from $p_i$ by removing two unit squares of opposite colors $s_{i}$ and $\tilde{s_i}$ at minimal distance.
Let $\gamma_{i}=(s_i , s_{i_1}, s_{i_2}, \ldots, s_{i_{d(s_i, \tilde{s_i})-1}}, \tilde{s_i})$ be the path in $T$ joining $s_i$ and $\tilde{s_i}$.
Therefore, each term of $\gamma_i$ corresponds to an unit square in $\D$ and any two consecutive terms correspond to a domino in $\D$.
Notice that, since $s_i$ and $\tilde{s_i}$ are of opposite colors, $\gamma_i$ has an even number of vertices.

Define the two reduced floors $f_{2i}^* = \{ s_{i_{2j-1}} \cup s_{i_{2j}} \colon j=1, 2, \ldots, \frac{d(s_i, \tilde{s_i})-1}{2} \}$ and $f_{2i+1}^*= \{s_i \cup s_{i_1}\} \cup \{ s_{i_{2j}} \cup s_{i_{2j+1}} \colon j=1, 2, \ldots, \frac{d(s_i, \tilde{s_i})-3}{2} \} \cup \{ s_{i_{d(s_i, \tilde{s_i})-1}} \cup \tilde{s_i}\}$.
In other words, $f_{2i}^*$ is formed by the consecutive dominoes along the path $\gamma_i \smallsetminus (s_i \cup \tilde{s_i})$ and $f_{2i+1}^*$ is formed by the consecutive dominoes along $\gamma_i$.
Consider the floors $f_{2i}=(p_i, \, f_{2i}^*, \, p_i^{-1} \smallsetminus f_{2i}^*)$ and $f_{2i+1}=(p_i^{-1} \smallsetminus f_{2i}^*, \, f_{2i+1}^*, \, p_{2i+1})$.
The tiling $\tl_p$ is described by the sequence 
$$ \tl_p = f_0 * f_1 * \ldots * f_{|p|}.$$
The important fact is that the projection on $\D$ of every horizontal domino in $\tl_p$ is an edge of $T$.
The Figure~\ref{fig:algorithmtiling} shows an example of the construction of $\tl_p$ for the disk $\D=[0,4]^2$.
\begin{figure}[ht] 
\centerline{
\includegraphics[width=0.9\textwidth]{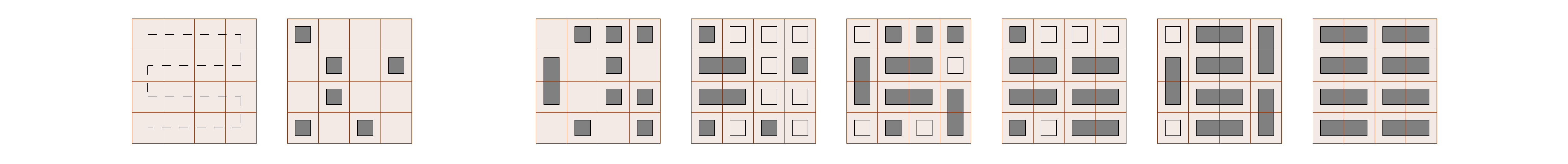}}
\caption{The disk $\D= [0,4]^2$ with a spanning tree, a plug $p$ and $\tl_p$.}
\label{fig:algorithmtiling}
\end{figure}

\begin{proof}[Proof of Theorem~\ref{thm:irregdisk-squrdisc}]
We first consider the case in which $\D \smallsetminus s$ has exactly three connected components $\D_1$, $\D_2$ and $\D_3$.
Suppose $|\D_1| \leq |\D_2| \leq |\D_3|$ so that, by hypothesis, $|\D_1 \cup \D_2| \geq 3$.
For $i \in \{1,2,3\}$ let $s_i \subset \D_i$ be a unit square adjacent to~$s$.
Moreover, let $d_1= s \cup s_1$, $d_2= s \cup s_2$ and $d_3=s \cup s_3$ be three dominoes.

We define two classes of floors $\mathcal{F}_0$ and $\mathcal{F}_1$.
A floor $(p_0,f^*,p_1)$ belongs to $\mathcal{F}_0$ if and only if:
\begin{enumerate}
    \item $f^*$ contains the domino $d_1$.
    \item $p_0$ marks all white unit squares in $\D_1 \smallsetminus s_1$ and all black unit squares in $\D_2$.
    \item $p_1$ marks all black unit squares in $\D_1 \smallsetminus s_1$ and all white unit squares in $\D_2$.
\end{enumerate}
A floor belongs to $\mathcal{F}_1$ if and only if its inverse belongs to $\mathcal{F}_0$; Figure~\ref{fig:thmdisksqrdisconnected} shows an example of a disk and its classes $\mathcal{F}_0$ and $\mathcal{F}_1$. 
This defines four classes of floors with parity: $\mathbf{f_0} = (\mathcal{F}_0,0)$, $\mathbf{f_1}=(\mathcal{F}_1,0)$, $\mathbf{f_0}^{-1}=(\mathcal{F}_1,1)$ and $\mathbf{f_1}^{-1}=(\mathcal{F}_0,1)$.

\begin{figure}[ht]
\centerline{
\includegraphics[width=0.46\textwidth]{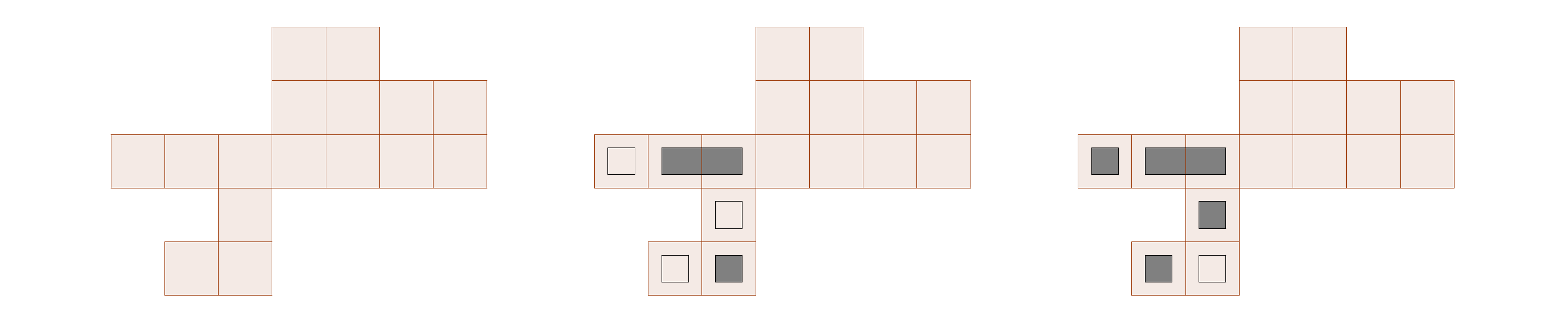}}
\caption{A disk and its two classes of floors $\mathcal{F}_0$ and $\mathcal{F}_1$.}
\label{fig:thmdisksqrdisconnected}
\end{figure}

We initially define a map $\Phi$ for oriented edges in $\mathcal{C}_{\D}^+$.
If $\mathbf{f}$ is a floor with parity not contained in the classes defined above, let $\phi(\mathbf{f}) = e$.
Otherwise, set $\Phi(\mathbf{f_0})=a$, $\Phi(\mathbf{f_1}) = b$, $\Phi(\mathbf{f_0}^{-1})=a^{-1}$ and $\Phi(\mathbf{f_1}^{-1})=b^{-1}$.

Consider two adjacent floors with parity in $\mathcal{C}_{\D}^+$ (i.e.\ a oriented path of length two) whose reduced floors contain the domino $d_1$.
We have only two possibilities.
First, both floors are neither in $\mathcal{F}_0$ nor in $\mathcal{F}_1$.
Second, either the first floor is in $\mathcal{F}_0$ and the second floor is in $\mathcal{F}_1$ or vice-versa.
Since adjacent edges have opposite parity, in any case we conclude that $\Phi$ maps this path of length two to the identity.
With this observation in mind, it is straightforward to check that $\Phi$ maps the boundary of any 2-cell to the identity.
Therefore, $\Phi$ extends to a homomorphism $\phi \colon G_{\D}^+ \to F_2$.

We now prove the surjectivity of $\phi$.
Suppose, without loss of generality, that $s$ is a white unit square.
Let $\tilde{s_3} \subset \D_3$ be a unit square adjacent to $s_3$.
Consider a floor $f=(p_0, \{d_1\} ,p_1)$ in $\mathcal{F}_0$ such that $p_1= p_0^{-1} \smallsetminus d_1$.
We may assume that $s_3 \not\subset p_0$ and $\tilde{s_3} \subset p_0$.
Notice that the two unit squares of opposite colors in $p_1$ at minimal distance are contained in $\D_2 \cup \D_3$.
Then, by construction, the floors of $\tl_{p_1}$ are contained neither in $\mathcal{F}_0$ nor in $\mathcal{F}_1$.
Let $p= (p_1 \smallsetminus s_3) \cup s_1$ be a plug and $g= (p, \{d_3\}, p_0)$ be a floor.
Analogously, the floors of $\tl_p$ are contained neither in $\mathcal{F}_0$ nor in $\mathcal{F}_1$.
Therefore, $\tl = \tl_p^{-1}*g*f*\tl_{p_1}$ is a tiling of $\D \times [0, 2|p_1| +2]$ such that $\phi(\tl)=a$.
Similarly, consider the plug $\tilde{p} = (p_0 \smallsetminus s_2) \cup s_1$ and the floor $\tilde{g}=(\tilde{p},\{d_2\},p_1)$.
Then, the tiling $\tilde{\tl} = \tl_{\tilde{p}}^{-1} * \tilde{g} * f^{-1} * \tl_{p_0}$  of $\D \times [0,2|p_0| +2]$ is such that $\phi(\tilde{\tl}) = b$.

Now, consider the case in which there exists a fourth connected component $\D_4$.
By possibly relabeling the components assume that $|\D_1| \leq |\D_2| \leq |\D_3| \leq |\D_4|$.
If $|\D_3| >1$ proceed as in the previous case: define two classes containing floors $(p_0, f^*, p_1)$ such that $d_1 \subset f^*$ and the plugs $p_0$ and $p_1$ mark alternately the unit squares in $\D_1 \cup \D_2 \cup \D_3$.

Suppose that $|\D_1|=|\D_2|=|\D_3|=1$.
We define two classes of floors $\mathcal{F}_0$ and $\mathcal{F}_1$.
A floor $(p_0, f^*, p_1)$ belongs to $\mathcal{F}_0$ if and only if $d_1 \subseteq f^*$, $s_2 \subset p_0$ and $s_3 \subset p_1$.
A floor belongs to $\mathcal{F}_1$ if and only if its inverse belongs to $\mathcal{F}_0$; see Figure~\ref{fig:thmdisksqrdisconnected2}.
As in the previous case, we have four classes with parity and a homomorphism $\phi \colon G_{\D}^+ \to F_2$.

\begin{figure}[H]
\centerline{
\includegraphics[width=0.58\textwidth]{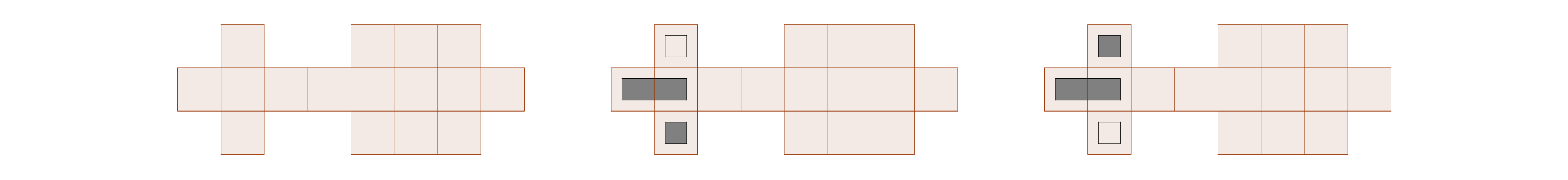}}
\caption{A disk and its two classes of floors $\mathcal{F}_0$ and $\mathcal{F}_1$.}
\label{fig:thmdisksqrdisconnected2}
\end{figure}

Let $f=(p_0,\{d_1\},p_1)$ be a floor in $\mathcal{F}_0$ such that $p_1= p_0^{-1} \smallsetminus d_1$.
Let $p = (p_1 \smallsetminus s_3) \cup s_1$ be a plug and $g=(p, \{d_3\}, p_0)$ be a floor.
Then, by construction, the floors of $\tl_{p_1}$ and $\tl_{p}$ are contained neither in $\mathcal{F}_0$ nor in $\mathcal{F}_1$.
Thus, for $\tl = \tl_p^{-1} * g * f * \tl_{p_1}$ we have $\phi(\tl) = a$.
Analogously, we obtain a tiling $\tilde{\tl}$ such that $\phi(\tilde{\tl})=b$.
\end{proof}

\begin{proof}[Proof of Theorem~\ref{thm:irregdisk-dominodisc}]
The proof consists of two cases.
We define, in both cases, two classes of floors $\mathcal{F}_0$ and $\mathcal{F}_1$.
The class $\mathcal{F}_1$ contains a floor if and only if its inverse belongs to $\mathcal{F}_0$, so that it suffices to define $\mathcal{F}_0$.

First consider the case in which there exists only one $2 \times 2$ square containing $d$; denote this square by $s$.
Let $\D_0$ be the connected component of $\D \smallsetminus d$ that intersects~$s$.
In this case, a floor $(p_0, f^*, p_1)$ belongs to $\mathcal{F}_0$ if and only if $d \subseteq f^*$ and $p_0$ (resp. $p_1$) marks all black (resp. white) unit squares in $\D_0$; as in Figure~\ref{fig:1.1}.
\begin{figure}[H]
\centerline{
\includegraphics[width=0.6\textwidth]{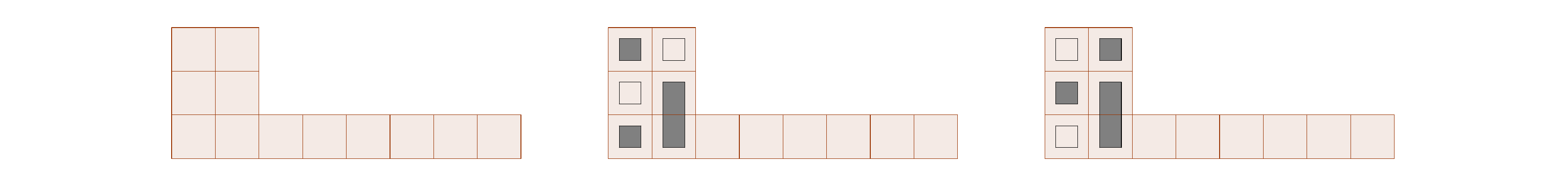}}
\caption{A disk and its two classes of floors $\mathcal{F}_0$ and $\mathcal{F}_1$.}
\label{fig:1.1}
\end{figure}

The second case is based on the existence of two distinct $2 \times 2$ squares $s_1, s_2 \subset \D$ such that $d \subset s_1$ and $d \subset s_2$.
Let $\D_1$ (resp. $\D_2$) be the connected component of $\D \smallsetminus s$ that intersects $s_1$ (resp. $s_2$).
Suppose that $|\D_1| \leq |\D_2|$.
In this case, a floor $(p_0, f^*, p_1)$ belongs to $\mathcal{F}_0$ if and only if:
\begin{enumerate}
    \item $f^*$ contains the domino $d$.
    \item $p_0$ marks all black squares in $\D_1$ and all white squares in $\D_2$.
    \item $p_1$ marks all white squares in $\D_1$ and all black squares in $\D_2$.
\end{enumerate}

\begin{figure}[H]
\centerline{
\includegraphics[width=0.45\textwidth]{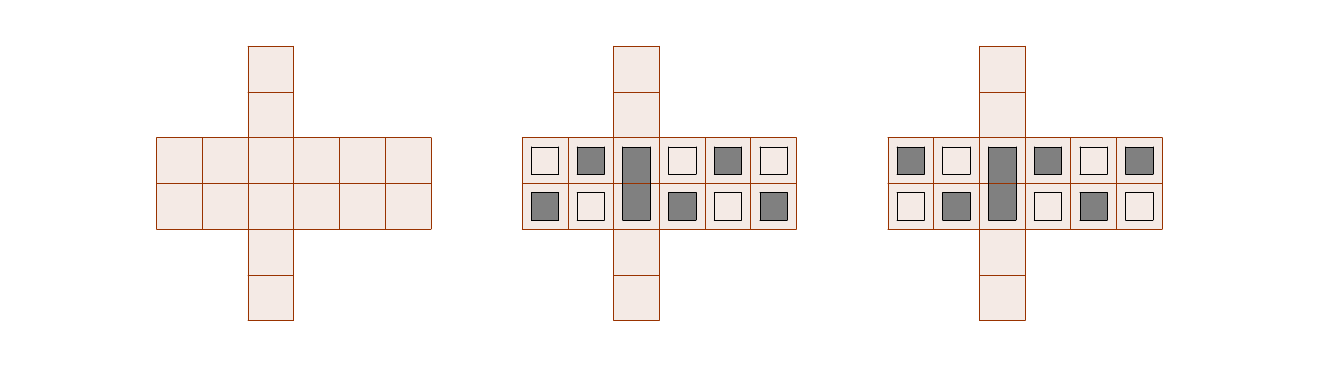}}
\caption{A disk and its two classes of floors $\mathcal{F}_0$ and $\mathcal{F}_1$.}
\label{fig:thmdominodisc2}
\end{figure}

Notice that the classes of floors are nonempty.
Indeed, by hypothesis, in the first case $|\D \smallsetminus d| - |\D_0| \geq |\D_0|$ and in the second case $|\D \smallsetminus d|-|\D_1|-|\D_2| \geq |\D_2|-|\D_1|$.
Since $\D$ is balanced, in both cases there exist plugs satisfying the required properties.

As in the proof of Theorem~\ref{thm:irregdisk-squrdisc}, we have four floors with parity which define $\Phi$ for oriented edges in $\mathcal{C}_{\D}^+$.
By construction, in a floor $(p_0, f^*, p_1)$ of class either $\mathcal{F}_0$ or $\mathcal{F}_1$, each connected component of $\D \smallsetminus d$ that intersects a $2 \times 2$ square which contains $d$ is marked alternately by $p_0$ and $p_1$.
Then, again as in the proof of Theorem~\ref{thm:irregdisk-squrdisc}, it is not difficult to check that $\Phi$ takes the boundary of any 2-cell to the identity.
Thus, $\Phi$ extends to a homomorphism $\phi \colon G_{\D}^+ \to F_2$.

We now prove that $\phi$ is surjective.
Consider a floor $f=(p_0,\{d_1\},p_1)$ in $\mathcal{F}_0$ such that $p_1= p_0^{-1} \smallsetminus d_1$.
Notice that, since $d$ is contained in a $2 \times 2$ square, there exists a spanning tree of $\D$ whose set of edges does not contain $d$.
Then, by definition, the tilings $\tl = \tl_{p_0}^{-1}*f*(p_1, \emptyset, p_1^{-1})* \tl_{p_1^{-1}}$ and $\tilde{\tl} = \tl_{p_1}^{-1}*f^{-1}*(p_0, \emptyset, p_0^{-1})* \tl_{p_0^{-1}}$ are such that $\phi(\tl)=a$ and $\phi(\tilde{\tl})=b$.
\end{proof}

\begin{proof}[Proof of Theorem~\ref{thm:dominodisc2}]
We define four classes of floors: $\mathcal{F}_0$, $\mathcal{F}_1$, $\mathcal{F}_2$, and $\mathcal{F}_3$.
Let $\mathcal{F}_0$ be the class of all floors $(p_0,f^*,p_1)$ such that $d_1 \subseteq f^*$ and $p_0$ (resp. $p_1$) mark all black (resp. white) unit squares in $\D_1$.
A floor belongs to the class $\mathcal{F}_1$ if and only if its inverse belongs to $\mathcal{F}_0$, i.e., $\mathcal{F}_1=\mathcal{F}_0^{-1}$.
Let $\mathcal{F}_2$ be the class of all floors $(p_0,f^*,p_1)$ such that $d_2 \subseteq f^*$ and $p_0$ (resp. $p_1$) mark all black (resp. white) unit squares in $\D_2$.
Finally, let $\mathcal{F}_3 = \mathcal{F}_2^{-1}$.

The Figure~\ref{fig:thm2item1} below shows an example of a disk, with $d_1$ and $d_2$ not parallel, and its four classes of floors.

\begin{figure}[H]
\centerline{
\includegraphics[width=0.5\textwidth]{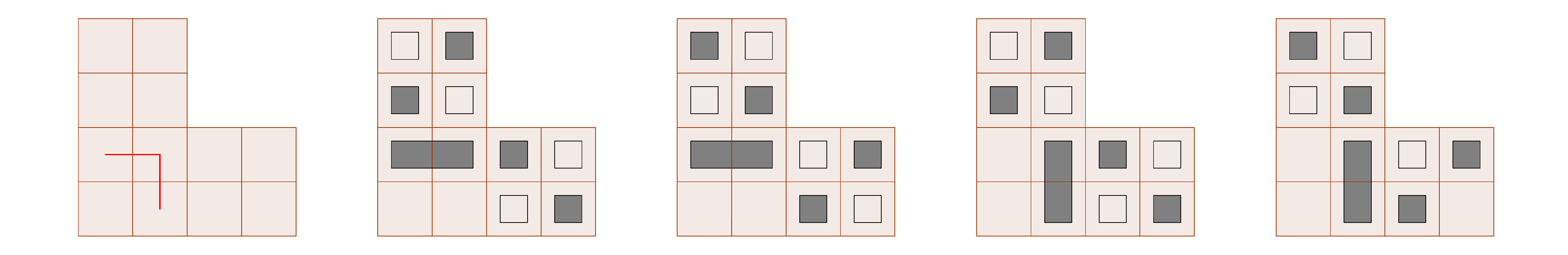}}
\caption{A disk $\D$, with $d_1$ and $d_2$ marked by a red line segment, and its 4 classes.}
\label{fig:thm2item1}
\end{figure}

If $d_1$ and $d_2$ are parallel then the intersections $\mathcal{F}_0 \cap \mathcal{F}_2$ and $\mathcal{F}_1 \cap \mathcal{F}_3$ are not empty, as shown in Figure~\ref{fig:thm222'}.

\begin{figure}[H]
\centerline{
\includegraphics[width=0.55\textwidth]{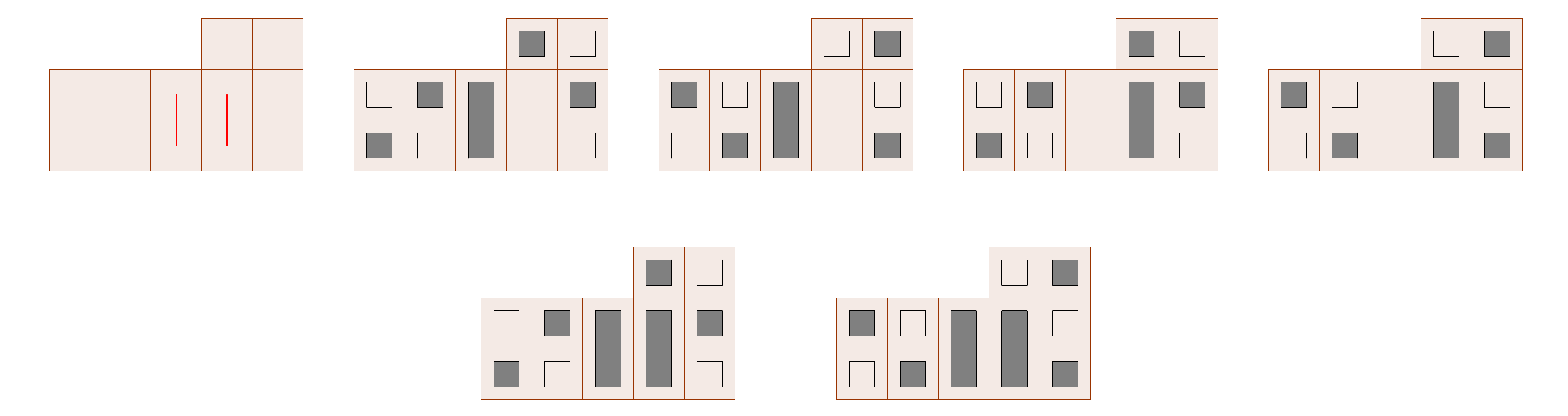}}
\caption{The first row shows a disk $\D$, with $d_1$ and $d_2$ marked by a red line segment, and its 4 classes.
The second row shows floors which are contained in the intersection of two classes.}
\label{fig:thm222'}
\end{figure}

We first define a map $\Phi$ for oriented edges in $\mathcal{C}_{\D}^+$.
Let $\mathbf{f}=(f,k)$ be a floor with parity. 
If $f$ belongs to either two or none of the classes constructed above, set $\Phi(\mathbf{f})=e$.
The image of the floors which are not taken to the identity depends on whether $d_1$ and $d_2$ are parallel.
Suppose that $f$ belongs to the class $\mathcal{F}_j$ for some $j=0, 1, 2, 3$.
If $d_1$ and $d_2$ are not parallel define $\Phi(\mathbf{f})$ as 
\begin{center}
$\begin{array}{c |c c c c c c c c c}
 (j,k) & (0,0) & (1,0) & (2,0) & (3,0) & (0,1) & (1,1) & (2,1) & (3,1) \\ 
\hline 
\Phi(\mathbf{f}) & a & b & b & a & b^{-1} & a^{-1} & a^{-1} & b^{-1} \\ 
\end{array}$
\end{center}
Otherwise, if $d_1$ and $d_2$ are parallel, define $\Phi(\mathbf{f})$  as 
\begin{center}
$\begin{array}{c |c c c c c c c c c}
 (j,k) & (0,0) & (1,0) & (2,0) & (3,0) & (0,1) & (1,1) & (2,1) & (3,1) \\ 
\hline 
\Phi(\mathbf{f}) & a & b & a^{-1} & b^{-1} & b^{-1} & a^{-1} & b & a \\ 
\end{array}$
\end{center}
In both cases, a careful analysis shows that $\Phi$ takes the boundary of any 2-cell to the identity.
Therefore, $\Phi$ extends to a homomorphism $\phi \colon G_{\D}^+ \to F_2$.
The surjectivity of $\phi$ follows as in the proof of Theorem~\ref{thm:irregdisk-dominodisc}.
\end{proof}

\begin{remark}
In many cases the homomorphisms constructed above are not isomorphisms.
The theorems are based on the existence of either a unit square $s$ or a domino $d$ that disconnects $\D$.
If one of the connected components of either $\D \smallsetminus s$ or $\D \smallsetminus d$ contains a $3 \times 2$ rectangle then there exists a tiling $\tl$ of $\D \times [0,4]$ such that $\phi(\tl)=e$ and $\textsc{Tw}(\tl)=1$; the kernel of $\phi$ then contains a nontrivial element.
The tiling $\tl$ is formed by taking a tiling of the $3 \times 2 \times 4$ box, as in Figure~\ref{fig:a}, and vertical dominoes outside the box. \hfill $\diamond$
\end{remark}
\begin{figure}[H]
\centerline{
\includegraphics[width=0.4\textwidth]{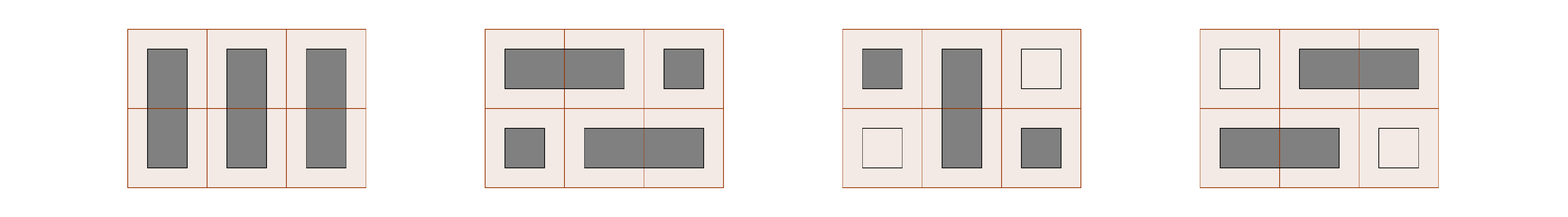}}
\caption{Tiling of a $3 \times 2 \times 4$ box.}
\label{fig:a}
\end{figure}



\noindent
\footnotesize
Raphael de Marreiros \\
Departamento de Matemática, Pontifícia Universidade Católica do Rio de Janeiro \\
Rua Marquês de São Vicente, 225, Gávea, Rio de Janeiro, RJ 22451-900, Brazil \\
\url{raphaeldemarreiros@mat.puc-rio.br}

\end{document}